\documentclass[11pt, reqno]{amsart}

\usepackage[top=1in,bottom=1in,left=1.2in,right=1.2in]{geometry}

\usepackage{appendix}
\usepackage{amsmath,amssymb,amsthm,amsfonts,enumerate, mathrsfs, mathtools}
\usepackage{relsize}
\usepackage{stmaryrd}
\usepackage{esint}
\usepackage{xcolor}
\usepackage{graphicx}
\usepackage{tikz}
\usetikzlibrary{patterns,positioning,arrows,shapes,trees}

\usepackage{url}
\usepackage[
colorlinks=true,
linkcolor=blue,
filecolor=blue,
urlcolor=red,
citecolor=magenta]{hyperref}

\usepackage{fancyhdr}
{}

\makeatletter
\newcommand{\mylabel}[2]{#2\def\@currentlabel{#2}\label{#1}}
\makeatother

\newcounter{counterConstant}
\newcommand{\const}[1]{
	\addtocounter{counterConstant}{1}
	\edef#1{\arabic{counterConstant}}
}

\numberwithin{equation}{section}
\def\theequation{\arabic{section}.\arabic{equation}}

\newtheorem{theorem}{Theorem}[section]
\newtheorem{lemma}[theorem]{Lemma}
\newtheorem{remark}[theorem]{Remark}
\newtheorem{definition}[theorem]{Definition}
\newtheorem{proposition}[theorem]{Proposition}

\newtheorem{assumption}{Assumption}

\def\cB{{\mathcal B}}
\def\cC{{\mathcal C}}

\def\cE{{\mathcal E}}
\def\cF{{\mathcal F}}

\def\cM{{\mathcal M}}

\def\cP{{\mathcal P}}

\def\cX{{\mathcal X}}

\def\mN{{\mathbb N}}

\def\mR{{\mathbb R}}

\def\mZ{{\mathbb Z}}

\def\bE{{\mathbf E}}

\def\bP{{\mathbf P}}

\def\sB{{\mathscr B}}

\def\l{\left}
\def\r{\right}
\def\<{\langle}
\def\>{\rangle}

\def\geq{\geqslant}
\def\leq{\leqslant}
\def\ge{\geqslant}
\def\le{\leqslant}

\def\1{{\mathbf{1}}}
\def\p{\partial}
\def\d{\mathrm{d}}
\def\e{\mathrm{e}}
\def\eps{\varepsilon}

\allowdisplaybreaks

\def\R{\mR}

\usepackage{soul}

\begin{document}
	
\title{Persistence and local extinction for superprocesses in random environments}
	
\author{Zhen-Qing Chen, \ Yan-Xia Ren \ and \   Guohuan Zhao}
	
\address{Department of Mathematics, University of Washington, Seattle, WA 98195, USA. }
\email{zqchen@uw.edu}

\address{LMAM School of Mathematical Sciences \& Center for
		Statistical Science, Peking University, Beijing, 100871, P.R. China.}
\email{yxren@math.pku.edu.cn}
	
\address{State Key Laboratory of Mathematical Sciences, Academy of Mathematics and Systems Science, CAS, Beijing, 100190, China}
\email{gzhao@amss.ac.cn}
	
\thanks{ The research of Zhen-Qing Chen is supported in part by a Simons Foundation grant.
The research of Yan-Xia Ren is supported by National Natural Science Foundation of China (No. 12231002)  and the Fundamental Research Funds for Central Universities, Peking University LMEQF. The research of Guohuan Zhao is supported by the National Key Research and Development Program of China (No. 2024YFA1013503) and the NSFC grants (Nos. 12271352, 12201611).}
	
\begin{abstract}
	We consider a super-Brownian motion $\{X_t, t\geq 0\}$ in a random environment described by a centered Gaussian field
	$\{W(t,x),t\geq 0, x\in\mathbb{R}^d\}$
	whose correlation function is given by $\mathcal{C} (x,y)(t\wedge s)$.  The process takes values in ${\mathcal M}(\mR^d)$, the space of Radon measures on $\mR^d$.  It can be characterized through  a conditional Laplace transform by a parabolic stochastic partial differential equation driven by $W(t, x)$.
	Suppose that $\mathcal{C} (x, y)\leq g(x-y)$ for some  bounded positive function $g$ on $\R^d$   and the initial distribution of
	process $X$ is the Lebesgue measure $m$ on $\mathbb{R}^d$. We  prove  that for dimension  $d\geq 3$, whenever
	$$
		\sup_{x\in \mathbb{R}^d} \int_{\mathbb{R}^d} |x-y|^{2-d} g(y)dy<  \frac{8 (d-2) \pi^{d/2}}{d 2^d \Gamma \left(d/2-1\right)},
	$$
	the distribution of $X_t$ converges weakly as $t \to \infty$ to a non-trivial invariant probability distribution $\pi^m$ on $\mathcal{M}(\mathbb{R}^d)$ with mean measure $m$. This result in particular gives an affirmative answer to {\bf Conjecture 1.4 } of  \cite{MX2007local}. We further show that given $ 
	\Theta \in C^\beta(\mathbb{R}^d)$ $(\beta>1)$, when $\cC (x,y)= a  
	 \Theta (x-y)$ with $a$ being large enough, the superprocess $X$ suffers local extinction.
	
\bigskip
 
	
\medskip
	

\end{abstract}
	
\maketitle
\tableofcontents

\section{Introduction}

The purpose of this work is twofold. Firstly, we give an affirmative answer to Conjecture 1.4  of  \cite{MX2007local} about the persistence property of $d$-dimensional
super-Brownian motion  with $d\geq 3$
 in ``weak" random environment constructed in \cite{Myt1996environments} and \cite{MX2007local}. Secondly, we sought to investigate which environmental properties influence and change the persistence/extinction properties of the same process.
	
\smallskip
	
Let us briefly review the probabilistic models with which we are concerned. In \cite{Myt1996environments}, Mytnik  introduced a family of  superprocesses in random environments. Let $\{\xi_k(\cdot): k\in\mathbb{N}\}$ be a sequence of independent identically distributed random fields on $\mR^d$ with
\begin{equation*}
	\bE \left[ \xi_k(x) \right] =0 \ \hbox{ and } \   \sup_{x\in \mR^d}\bE \left[ |\xi_k(x)|^3 \right] <\infty
	\quad \hbox{for every } x\in \mathbb{R}^d     \hbox{ and }       k\in\mathbb{N}.
\end{equation*}
It follows that the   correlation function
$$\cC (x, y):=\bE \left[ \xi_k(x)\xi_k(y) \right]$$ is bounded on $\R^d\times \R^d$. The random fields $\{\xi_k(x): x\in \mR^d, k\in\mN\}$ serve as random environments for the following stochastic models. For each fixed $n\in \mathbb{N}$, define
$$
\xi_k^{(n)}(x):=\left\{\begin{array}{ll} \sqrt{n},\quad & \xi_k(x)>\sqrt{n},\\
-\sqrt{n}, & \xi_k(x)<-\sqrt{n},\\
\xi_k(x), & \mbox{otherwise}.\end{array}\right.
$$
Consider a particle system in which there are $K_n$ particles  in $\mR^d$, each of them moves independently  as a Brownian motion until time $t=1/n$. Given $\{\xi_k(x): x\in \mR^d, k\in\mN\}$, at time $1/n$, 
each particle located at $x$
splits into two particles with probability $\frac{1}{2}+\frac{1}{2\sqrt{n}} 
	\xi_1^{(n)}(x)$
	or dies with probability
	$\frac{1}{2}-\frac{1}{2\sqrt{n}} 
	\xi_1^{(n)}(x)$.
	The new particles then moves in space independently as  standard  Brownian motions starting at their places of birth,
	during the time interval $[1/n, 2/n)$. In general,  at time $i/n$, each
	surviving  particle  located at $x$ splits into two particles with probability
	$\frac{1}{2}+\frac{1}{2\sqrt{n}} \xi_i^{(n)}(x)$
	or dies with probability
	$\frac{1}{2}-\frac{1}{2\sqrt{n}} \xi_i^{(n)}(x)$,
	and in the time interval $[i/n, (i+1)/n)$ particles independently according to Brownian motions.
	Let $X^n_t$ be the empirical measure-valued Markov process, defined as
	$$
	X^n_t(A)=\frac{\hbox{number of particles in }A \hbox{ at time }t} {n}  \quad \hbox{ for }  A \in \sB( \mR^d),
	$$
	where $\cB(\R^d)$ denotes the Borel $\sigma$-field on $\R^d$.
	It is proved by Mytnik \cite[Theorem 3.1]{Myt1996environments} that if $X^n_0$ converge weakly to a finite measure $\nu$
	on $\mR^d$, then the processes $X^n=\{X^n_t, t\ge 0\}$ converges weakly to a  finite-measure-valued process
	$X=\{X_t,t\ge 0\}$, where $X$ is the unique solution to the following martingale problem:
	for every $f\in C^2_c(\mR^d)$,
	\begin{equation} \label{eq:MP}
		\begin{aligned}
			& M_t^f   := \langle f, X_t\rangle-\langle f, \mu\rangle-\frac{1}{2}\int_0^t \l\langle \Delta f,  X_s\r\rangle \d s
			\ \mbox{ is a continuous martingale}\\
			&\mbox{ with quadratic variation }
			\langle M_t^f \rangle_t= \int_0^t \langle f^2, X_s\rangle \d s+ \int_0^t \langle \cC\cdot f\otimes f, X_s\otimes X_s\rangle \d s,
		\end{aligned} \tag{{\bf MP}}
	\end{equation}
	where $(f\otimes h)(x,y): =f(x)h(y)$  and $C^2_c(\R^d)$ is the space of twice continuously differentiable functions with compact support.
	 Later, Mytnik-Xiong \cite{MX2007local} introduced  a conditional martingale problem  with noise $W$, and characterized the conditional Laplace transform of
	 solutions of \eqref{eq:MP}.
	The  random noise  $W$ is a Gaussian field with zero mean that is white in time and colored in space with
	covariance function
	$$
	    \bE\left[ W(s,x)W(t,y) \right]=\cC (x,y)(s\wedge t).
	$$

\medskip

For $\rho >0$, set
\begin{equation}\label{def-phi}
    \phi_\rho(x):=(1+|x|^2)^{-\rho/2}.
\end{equation}
Denote by $\cM(\mR^d)$  the space of Radon measures on $\mR^d$ equipped with vague topology  and
$$
\cM_\rho(\mR^d):=\{\mu\in \cM(\mR^d): \langle \phi_\rho, \mu\rangle<\infty\}.
$$	

The following definition of conditional martingale problem is essentially taken from  \cite{MX2007local}.
\begin{definition}\label{D:1.1}
Let $X$ be a measure-valued process on $\mR^d$ defined on the probability filtered
space $(\Omega,\mathcal{F},  \mathcal{F}_t, \bP)$   so that $X_t\in \cF_t$ for every $t\geq 0$ and  $W=\{W (t, x); t>0, x\in \mR^d\}$  is
a centered Gaussian field defined on $(\Omega,\mathcal{F},  \bP)$ with
		$$
		 \bE \l[W(s,x)W(t,y)\r] =\cC (x,y) (s\wedge t) \quad\mbox{ for any }s,t\geq 0 \mbox{ and x,y }\in \mR^d.
		$$
Let $\cF^W_\infty = \sigma(\{W(s,x) : s \geq 0,x \in\R^d\})$.
Denote by $\bP^W$ the conditional probability given $W$ and define $\mathcal{G}_t=\mathcal{F}_t\vee\mathcal{F}_\infty^W$. 
For $\mu\in \cM(\mR^d)$, we say $X$ is a solution to  conditional martingale problem with respect to $W$ with initial value $X_0 = \mu$ if for every $f\in C^2_c (\mR^d)$,
		\begin{equation}\label{eq:CMP}
			\begin{aligned}
				&N_t^f=\langle f, X_t \rangle-\langle f, \mu\rangle- \frac{1}{2} \int_0^t\langle \Delta f, X_s\rangle \d s-\int_0^t\!\!\int_{\mR^d}f(x)W(\d s,x)X_s(\d x) \\
				&\mbox{ is a continuous }
				\{\mathcal{G}_t\}\mbox{-martingale under $\bP^W$ with quadratic variation }\\
				&\langle N^f\rangle _t = \int_0^t \langle f^2, X_s\rangle \d s  \  \hbox{ and } \  N_0^f=0.
			\end{aligned}\tag{{\bf CMP}}
		\end{equation}
	\end{definition}

\begin{remark} \label{rem:well-posedness} \rm
 The well-posedness of the martingale problem \eqref{eq:MP} for finite initial measures is established in  \cite[Theorem 3.1]{Myt1996environments}.
	Although it is claimed on p.1960 of \cite{Myt1996environments} that the above can be extended to the case where the initial measure is the Lebesgue
	measure on $\R^d$, which is  subsequently utilized in the proof of \cite[Corollary 2.4]{MX2007local}, 	
	a rigorous treatment of the well-posedness of the martingale problem \eqref{eq:MP} for infinite initial measures  is  
	currently still lacking in the literature. 
	Specifically, the uniqueness proof in \cite[Section 4]{Myt1996environments} constructs a dual process taking values in $C_0(\mR^d)$ rather than $C_c(\mR^d)$,
	which works only for  the finite initial measure case. Here $C_0(\R^d)$ (resp. $C_c(\R^d)$) is the space  of all continuous functions on $\R^d$
	that vanishes at infinity (resp. have compact support).
	To accommodate infinite initial measures, the framework must be refined by restricting the initial measures to a suitable class, such as $\cM_\rho(\mR^d)$.
	In this paper, we address these gaps in literature as follows:
    \begin{enumerate}[\rm (i)]
    \item In Appendix \ref{Sec:App-MP}, we provide a proof for the uniqueness of solutions to \eqref{eq:MP} for
     any initial measure in $\cM_\rho(\mR^d)$ with $\rho>0$, which contains the Lebesgue measure on $\R^d$ as a particular member. 
      This validates  the proof of and hence the application of
    \cite[Corollary 2.4]{MX2007local} in this setting.

    \item In Proposition \ref{prop:CMP}, we directly establish the  existence and uniqueness for solutions to \eqref{eq:CMP} (viewed as \(\cM_{\rho}(\mR^d)\)-valued continuous processes), without resorting a priori to
    the equivalence between \eqref{eq:MP} and \eqref{eq:CMP}  with infinite initial measures in $\cM_\rho(\mR^d)$.
    \end{enumerate}
\end{remark}

Using \cite[Lemma 2.3 and Corollary 2.4]{MX2007local} and Proposition \ref{prop:CMP} below, we see that both \eqref{eq:MP} and \eqref{eq:CMP} are well-posed for initial measures in $\cM_\rho(\mathbb{R}^d)$, and the two problems are equivalent in the sense that the law of the solution to one problem coincides with that of the other.
 In view of these equivalences, and noting that solutions to \eqref{eq:MP} arise as scaling limits of branching Brownian motion in random environments, we shall henceforth refer to any solution of \eqref{eq:MP} or \eqref{eq:CMP} as a super-Brownian motion in random environment (SBMRE).

\medskip	
	
It is established in \cite[Corollary 2.4 and Theorem 2.15]{MX2007local}
(in light of Theorem \ref{T:A.4} of this paper)
that for any
\(\mu\in \cM_\rho(\mR^d)\) and $f\in C_c(\mR^d)$, the conditional Laplace transform of  a solution $X_t$ of \eqref{eq:CMP} with $X_0=\mu$ is characterized by
	\begin{equation}\label{eq:clt}
		\bP_\mu^W \exp(-\langle f, X_t\rangle):=\bP_\mu [\exp(-\langle f, X_t\rangle)|W]=\exp(-\langle u(t,\cdot),\mu\rangle),
    \end{equation}
	where
	$u$ is the pathwise unique solution to the  following stochastic partial differential equation (SPDE):
	\begin{equation}\label{eq:spde-clt}
		\p_t u= \frac{1}{2} \Delta u- \frac{1}{2} u^2+ u\, \p_t W, \quad u(0)=f.
	\end{equation}
	 The solution is understood in analytically weak and probabilistically strong sense.

	With the help of conditional Laplace transformation, Mytnik-Xiong \cite{MX2007local} proved that if
	\begin{equation}\label{eq:cond-cC}
		c(|x-y|^{-\alpha}\wedge 1)\leq \cC (x,y)\leq C
		\quad \hbox{for }  x, y\in \mathbb{R}^d,
	\end{equation}
	for some $c$, $C>0$ and $0\leq \alpha\leq 2$, then starting form the Lebesgue measure $m$ on $\mR^d$, the process $X_t(K)\rightarrow 0$ in probability,
 for any compact subset $K$ of $\mathbb{R}^d$. This means in the presence of a heavy-tailed correlation function $\cC$, SBMRE exhibits weak local extinction for any $d\geq 1$.
	
	On another hand, based on Dawson-Salehi \cite{DS1980spatially}, Mytnik and Xiong conjectured in  \cite[Conjecture 1.4]{MX2007local}
	   that if $\cC (x,y)\leq \eps (|x-y|^{-\alpha}\wedge 1)$ with $\alpha>2$, $\eps$ sufficiently small and $d\geq 3$,
	 SBMRE $X$ starting from the Lebesgue measure $m$ on $\mR^d$ will converge weakly to some non-trivial random measure $X_\infty$ on $\mR^d$ as $t\to \infty$.
		Here the assumption on $\cC $  means the environments are \textit{ ``weakly" correlative} in the sense that $\cC (x, y)$ goes to $0$ sufficiently fast  as  $|x-y|\to \infty$ and the variance  $\bE  \left[ \xi (x)^2 \right] = \cC(x, x)$ of the environment is sufficiently small for all $x\in \mathbb{R}^d$.

	We should mention that Rosati-Perkowski \cite{PR2021rough} recently studied the scaling limit of a branching random walk in static random environment in dimension $d=1,2$, and showed that it is given by a super-Brownian motion in a white noise potential  (that is, with $\cC(x,y)=\delta(x-y))$). In both dimensions 1 and 2,
	they prove the persistence property of this super-Brownian motion in rough  random  environment, which is
 opposite to what happens for the classical super-Brownian motion.
	
	\medspace
	
	\subsection{Main results}
    In this paper, we consider the limit  behaviors of the SBMRE $X$
   considered in \cite{MX2007local} starting from the Lebesgue measure $m$. In particular, we give  an affirmative answer to Conjecture 1.4 of  \cite{MX2007local}. We then further study the long-term behavior of  SBMRE when the correlation function of the random field is large near the diagonal (but there is no requirement on the decay speed of the correlation function as $|x-y|\to\infty$).

    \smallskip

    Throughout this work, 
    we  always assume
    $$
      \bE \l[W(s,x)W(t,y)\r]=\cC (x,y) (s\wedge t) \, \mbox{ with }\,  \|\cC\|_\infty<\infty.
    $$
   We consider SBMRE starting from
an infinite measure $\mu \in \cM_\rho(\mR^d)$ with $\rho > 0$ as a solution to the conditional martingale problem \eqref{eq:CMP}. In Proposition \ref{prop:CMP}, we establish  the existence and
uniqueness of solutions to equation
\eqref{eq:CMP} when $X_0 =\mu\in \cM_\rho(\mR^d)$, 
 which includes the Lebesgue measure $m$ on $\R^d$ as a special case provided $\rho>d$. 
The family of probability
measures on the space $\mathcal{M}_\rho(\mathbb{R}^d)$ will be denoted by  $\cP(\mathcal{M}_\rho(\mathbb{R}^d))$.

The following is the first main result of this paper.

  \begin{theorem}\label{thm:1}
	Let $\rho>d\geq 3$. Assume that
	\begin{equation}\label{cond:C}
    	\begin{aligned}
    		\cC (x,y)\leq&  g(x-y) ~ \mbox{ for some  bounded non-negative function }g
    	\end{aligned} \tag{\bf{H$_\cC$}}
    \end{equation}
    and
	\begin{equation}\label{cond:g}
		\theta:= \sup_{x\in \mR^d} \int_{\mathbb{R}^d} |x-y|^{2-d} g(y)\d y< \frac{8 (d-2) \pi^{d/2}}{d 2^d \Gamma \left(d/2-1\right) }.  \tag{{\bf H$_g$}}
	\end{equation}
	Then the distribution of $X_t$ starting from $m$ converges weakly to a non-trivial distribution $\pi^m$ in $\cP(\cM_\rho(\mR^d))$ as $t\to\infty$
	with mean $\int \mu\pi^m(\d \mu)=m$. Consequently, $\pi^m$ is an invariant probability distribution of $X_t$.
  \end{theorem}

\begin{remark} \rm
\begin{enumerate} [\rm (i)]
    \item Let $d\geq 3$, and $\theta$ be the constant given by \eqref{cond:g}. When $W$ is spatially homogeneous and  $\cC(x,y)=g(x-y)$ satisfies \eqref{eq:cond-cC} with $\alpha\in [0,2]$, then $\theta=\infty$.

\smallskip

    \item If $0\leq g(y)\leq \frac{\eps}{1+|y|^{\alpha}}$ for some $\alpha>2$, then
   	   \begin{align*}
   		 \theta
		  :=  & \sup_{x\in \mR^d} \int_{\mathbb{R}^d} |x-y|^{2-d} g(y)\d y
   		 \leq  \sup_{x\in\mR^d} \int_{\mR^d} \frac{\eps }{(1+ |y|^{\alpha}) |x-y|^{d-2}} dy \\
    	 \leq& \int_{\mR^d} \frac{\eps }{(1+ |y|^{\alpha}) |y|^{d-2}} dy
	 < \frac{8 (d-2) \pi^{d/2}}{d 2^d \Gamma \left(d/2-1\right)},
   	   \end{align*}
   	   provided that $\eps$ is sufficiently small. Therefore, Theorem \ref{thm:1} in particular gives an affirmative answer to \cite[Conjecture 1.4]{MX2007local}.
	
\smallskip

	\item  It is recently shown in \cite{FHX2025quenched} that for $d\geq 3$ and $\alpha >2$,  there is a positive constant $\eps_{0}=\eps_{0}(d, \alpha)$ so that when
		\begin{equation} \label{e:1.4}
	    \cC (x,y)\leq g(x-y) = \frac{\eps}{1+|x-y|^{\alpha}}
	\end{equation}
	with $0<\eps <\eps_{0}$,
	then $\frac{1}{T}\int_0^T \<f, X_t\> \d t$ converges to $ \<f,m\>$ as \(T\to\infty\) in both quenched and annealed sense   for any \(f\in C_c(\mR^d)\). As we noted in  part (ii) of this remark, condition \eqref{e:1.4} is a special case under which conditions \eqref{cond:C}-\eqref{cond:g} hold. Our Theorem \ref{thm:1} establishes that, under  conditions \eqref{cond:C}-\eqref{cond:g},
 the process \(X_t\) itself converges weakly to a non-trivial random measure $\pi^m$ with mean \(m\) as \(t\to\infty\).
\end{enumerate}
\end{remark}

  As mention before, it is proved in \cite[Theorems 1.1 and 1.2]{MX2007local}   that  if the covariance $\cC(x, y)$ of
  $\xi_k(x)$ and $\xi_k(y)$
  does not  decay too fast as $|x-y|\to \infty$, then the process $X_t$ suffers local extinction in weak sense. The following theorem
  is the second main result of this paper, which
   gives another sufficient condition for almost sure  local extinction.

  \begin{theorem}\label{thm:2}
  	Let $d\geq 1$. Suppose that   $a>0 $ and
	  	 \begin{equation}\label{cond:Gamma}
  	 	\begin{aligned}
  \cC(x, y)=a \Theta(x-y)   \quad \hbox{ for some } \Theta\in C^\beta(\mR^d) \mbox{ with } \beta>1 \hbox{ and } \Theta(0)=1.
	 	\end{aligned}
  	 	 \tag{{\bf H$_\Theta$}}
  	\end{equation}
  	There exists  a constant $N_0:=N_0(d, \Theta)\geq 1$ such that if $a\geq N_0$, then for any compact set $K\in \mR^d$ there exists a random time $T_K<\infty$ a.s. so that
  	$X_t(K)=0$ for every $ t\geq T_K$.
  \end{theorem}

  \begin{remark} \rm
   \begin{enumerate}[(i)]
     \item Roughly speaking, Theorem \ref{thm:2} indicates that when the variance of $\xi_k(x)$ is substantial, regardless of how rapidly the correlation function $\cC(x,y)$ diminishes as $|x-y|\to\infty$, the SBMRE is almost surely to experience local extinction.

 \smallskip

     \item An interesting question is whether SBMRE also adheres to the extinction/persistence dichotomy.
    	 Suppose that the spatial covariance function for $W(t,x)$ is given by $\cC(x,y)=g(x-y)$. Let $\theta$ be the constant given by \eqref{cond:g}.  Drawing from the insights provided by the main results in \cite{MX2007local}
      and our two main results above, we conjecture
      there exists a critical index, denoted as $\theta_c$, such that when $\theta>\theta_c$, $X$ suffers local extinction
     and when $\theta<\theta_c$, $X_t$ converges weakly to a nontrivial random measure on $\mR^d$.
   \end{enumerate}
  \end{remark}

   \subsection{Strategy of  our approach}

   Let  $v$ be the uniqueness mild solution to the following parabolic Anderson model (PAM) in $\mR^d$:
   \begin{equation}\label{eq:PAM}
   	\p_t v=\frac{1}{2} \Delta  v+v\,  \p_t W   \quad \hbox{with } v(0)=f.
   \end{equation}
   As in \cite{MX2007local},  we
   employ the conditional Laplace transformation and the comparison principle for semilinear SPDEs to examine the limiting characteristics of SBMRE. This entails investigating the asymptotic behavior of $v(t,x)$ as $t$ approaches infinity. Demonstrating the persistence of $X$ under condition \eqref{cond:g} necessitates more refined estimates for
   $$
     \bE \int_0^\infty\!\!\int_{\mR^d} v^2(t,x) \d x\d t ~\mbox{ and }~ \bE \int_0^\infty\!\!\int_{\mR^{2d}} v(t,x)\cC(x,y) v(t,y) \d x \d y \d t.
   $$
   To establish the boundedness of the above two quantities, we undertake an analysis of the
estimates for the Green function associated with a Feynman-Kac semigroup defined by the
function $g$. These estimates are elaborated in the proof of Lemmas \ref{lem:intEv2} and \ref{lem:intEvv} below.

   \medskip
   The proof of finite-time local extinction of X under condition \eqref{cond:Gamma} requires a more detailed
examination of the rate at which $v$ converges to zero as time approaches infinity. In fact, PAM \eqref{eq:PAM} with colored Gaussian noise was studied in a number of papers (see e.g. \cite{CV1998almost}, \cite{FV2006Lyapunov} and \cite{CM2006lyapunov}). Under the setting of Theorem \ref{thm:2}, it can be deduced from \cite[Theorem 2]{CV1998almost} that there is a constant $N_0>0$ so that when $a\geq N_0$ and $t\gg 1$, for each $x\in \mR^d$, it holds that $\log v(t,x)/t \leq -{a}/{3}$ with high probability.
  In fact, we show that there is constant $N_0>0$
   so that for
    $a\geq N_0$ and $t\gg 1$,
   $$
     \log \l(\sup_{|x|\leq t}v(t,x) \r)\Big/t \leq -a/3 \quad \mbox{with probability at least } 1-\e^{-c t}.
   $$
   Its proof will be given in Appendix \ref{Sec:App-LDP} of this paper, where we utilize a Harnack-type inequality and a chaining argument. Armed with this estimate, to obtain the local extinction result, we employ a Borel-Cantelli argument, following the approach outlined in \cite{MX2007local}.

   \medskip
   \begin{remark} \rm While the paper is just finalized now, a major portion of its work was
     completed some time ago.
   The results had been reported at several conferences. A preliminary version of this paper was made available
   to some scholars in the field upon their requests, and it had been cited  in  \cite{FHX2025quenched} and   \cite{CHY25occupation}. 
       \end{remark}

   \medskip

  The rest of this paper is organized as follows.
  In Section \ref{Sec:TR},  we prepare some preliminary results that will be used later.
  In Section \ref{Sec:MF}, we give the first and second moments formulas for SBMRE starting from finite initial measure. Section \ref{Sec:persistence} is dedicated to addressing the existence and uniqueness of solutions to equation \eqref{eq:CMP} with $\mu\in {\cM}_\rho(\mR^d)$ and proving Theorem \ref{thm:1}. Section \ref{Sec:extincition} is focused on providing the proof of Theorem \ref{thm:2}. Some results about positivity and comparison principle for semilinear SPDEs are stated in Appendix \ref{Sec:App-SPDE}. A  large deviation result, which plays a crucial role in the proof for Theorem \ref{thm:2} is presented in Appendix  \ref{Sec:App-LDP}.

  \smallskip

  We close this section by mentioning some conventions and notations used throughout this paper:
  \begin{itemize}
  	\item we use  $:=$ as a way of definition.
  	
  	\item For $a, b\in \mathbb{R}$, $a\vee b:= \max \{a, b\}$ and $a\wedge b :=\min \{a, b\}$.

  	\item For  a measurable function $f$,  $f^{+}:= f \vee 0$ and $f^{-}:=(-f) \vee 0$
	 stand  for its positive and negative part, respectively.
  	
  	\item For any measurable functions $f_i$ defined on $\cX_i$ for $i=1,2$, let $f_1\otimes f_2$ denote the real-valued function on $\cX_1\times \cX_2$ defined by  $(f_1\otimes f_2)(x,y):=f_1(x_1)f_2(x_2)$ for all $(x_1,x_2)\in \cX_1\times \cX_2$.
  	
  	\item $\mR_+:=(0,\infty)$ and $\overline{\mR}_+:=[0,\infty)$.
  	
  	\item $B_r(x):=\{y\in \mR^d: |y-x|<r\}$ is the open ball with radius $r>0$ centered at $x\in \mR^d$.

  	\item Let $\cX$ be a metric space.
  	\begin{itemize}
  	    \item $\sB(\cX)$: the collection of all Borel subsets of $\cX$.
  	
  		\item $\mathcal{B}(\mathcal{X})$ [resp. $\mathcal{B}_b(\mathcal{X})$]: the space of all [resp. bounded] Borel measurable functions on $\mathcal{X}$; the superscript '$+$' restricts these to non-negative functions.
  		
  		\item $C_b(\cX)$ ($C_b^+(\cX)$):  the space of all (positive) bounded continuous functions on $\cX$.
  		
  		\item $C_c(\cX)$ ($C^+_c(\cX)$):  the space of all (positive) continuous functions on $\cX$ with compact support.
  		
  	\end{itemize}
  	
  	\item For any $u,v\in L^2(\mR^d)$,
  	$$
  	(u, v):= \int_{\mR^d} u(x) v(x) \d x,  \quad \| u\|_2^2:=(u, u).
  	$$
  	\item
  	Set \(H^1(\mR^d)\) as the usual Sobolev space defined by
  	$$
  	H^1(\mR^d):=\left\{u\in L^2(\mR^d): \nabla u\in L^2(\mR^d; \mR^d)\right\}.
  	$$
  	For any $u\in H^1(\mR^d)$, $\| u\|_{H^1}^2 :=  \| u\|_2^2 +\| \nabla u\|_2^2$.
  	
  	\item $\mathcal{M}(\mathbb{R}^d)$: the space of Radon measures on $\mR^d$ equipped with vague topology.
  	
  	\item $\mathcal{M}_F(\mathbb{R}^d)$: the space of finite measures on $\mR^d$ equipped with weak topology.
  	
  	\item $\cM_\rho(\mR^d):=\{\mu: \mu\in\mathcal{M}(\mathbb{R}^d)\mbox{ such that }\langle \phi_\rho, \mu\rangle<\infty\}$.
  	
  	\item $\cP(\mathcal{M}(\mathbb{R}^d))$: the set of all probability measures on $\mathcal{M}(\mathbb{R}^d)$.
  	
  	\item Throughout this paper, we use $C$ or $c$ to denote positive constants whose exact value is not important and can change from line to line, and also use the notation $C(a, b, \cdots)$ or $c(a,b,\cdots)$ to denote a constant which depends only on the numbers $a,b,\cdots$.
  \end{itemize}
  In this paper, we use  $p(t,x):=(2\pi t)^{-d/2} \e^{-\frac{|x|^2}{2t}}$ to denote the transition probability density of Brownian motion in $\mR^d$, and
  $\{P_t; t\geq 0\}$ to denote the transition semigroup associated with Brownian motion,
 that is, 
  for $f\in \cB^+ (\R^d)$,
  $P_0f=f$ and
  $$
  P_t f (x) = \int_{R^d} p(t, x-y) f (y)dy \quad \hbox{for } t>0 \hbox{ and }x\in \R^d.
  $$
   It is well known that when $d\geq 3$, the Green function of $\Delta$ is given by
  $$
  G (x, y):=\int_0^\infty p(2t, x-y) \d t = \frac{\Gamma\left(d/2-1\right)}{4 \pi^{d/2}} |x-y|^{2-d}.
  $$

  \section{Some Technical  Results}\label{Sec:TR}

   Here and in the subsequent discussions,
  we assume that $B_t$ and $B'_t$ are  two independent $d$-dimensional Brownian motions starting respectively from $x,y\in\R^d$ under $P_{(x,y)}$. We  use $E_{(x,y)}$ to 
   denote expectation with respect to $P_{(x,y)}$.

  \subsection{Moment estimates for PAM}

  We need some moment estimates for the mild solution
$v(\omega,t,x)$ of PAM \eqref{eq:PAM}. Note that $v$ is random and is defined on $\Omega\times[0,\infty)\times\R^d$, but unless
otherwise specified, we typically suppress $\omega$ from the notation $v(\omega,t,x)$ and denote it simply
by $v(t,x)$.  An explicit representation for $\bE[v(t,x)v(t,y)]$ is given in the following lemma.
  \begin{lemma}\label{lem:Evv}
  	Suppose $f\in C_c^+  (\mR^d) $
  	and $v$ is the unique solution to \eqref{eq:PAM}.  For each $(t, x, y)\in (0, \infty)\times \mR^d \times \mR^d$,
  	\begin{align}\label{eq:E-v2}
  			\bE   \left[v(t,x)v(t,y)\right]
  			= E_{(x,y)} \left[ \exp\left(\int_0^t \cC (B_s, B'_s)ds\right) (f\otimes f)(B_t, B'_t)\right].
	\end{align}
  \end{lemma}

  \begin{proof}
  	 Since $v$ is a mild solution to
  \eqref{eq:PAM},
  	$$
  	    v(t,x)=P_tf(x)+\int_0^t\!\!\int_{\mathbb{R}^d}p(t-s,x-z)v(s,z)W(\d s,z)\d z.
  	$$
  	By It\^o isometry,
  	\begin{equation} \label{eq:E-square-v}
  		\begin{aligned}
  			&h(t,x,y):=\bE  \left[ v(t,x)v(t,y) \right] \\
  			=&P_tf(x)  P_tf(y)+\bE  \left[ \int_0^t\!\!\int_{\mathbb{R}^d}p(t-s,x-z)v(s,z)W(\d s,z)\d z \right.\\
  			&  \qquad \qquad  \qquad  \qquad  \qquad    \left. \int_0^t\!\!\int_{\mathbb{R}^d}p(t-s,y-w)v(s,w)W(\d s,w)\d w\right] \\
  			=&P_tf(x)  P_tf(y)
  			+\int_0^t \d s\!\! \int_{\mathbb{R}^{2d}}p(t-s,x-z)p(t-s,y-w)\bE  [v(s,z)v(s,w)]\cC (z,w)\d z\d w\\
  			=&Q_t(f\otimes f)(x,y)+\int_0^t Q_{t-s}(\cC (\cdot, \cdot) h(s, \cdot, \cdot))(x,y)ds,
  		\end{aligned}
  	\end{equation}
  	where $Q_tF (x,y):=\int_{\mR^{2d}} p(t,x-z)p(t,y-w) F(z,w)\d z \d w$  for any non-negative function $F$
defined on $\R^d\times\R^d$.
  	Equation \eqref{eq:E-square-v} means that  $h(t,x,y)$ is a solution on $\R^{2d}$ of
  	$$
  	  \partial_t h= \frac{1}{2}\Delta h+\cC h, \quad h(0)=f\otimes f.
  	$$
  	It follows from the Feynman-Kac representation that
  	\begin{eqnarray*}
  		\bE   \left[v(t,x)v(t,y)\right]=h(t, x, y)
  		= E_{(x,y)} \left[ \exp\left(\int_0^t \cC (B_s, B'_s)\d s\right) (f\otimes f)(B_t, B'_t)\right] .
  	\end{eqnarray*}
    This proves our assertion.
  \end{proof}

  One can also derive explicit expressions about the $n$th moment $\bE [v^n(t,x)]$ following the
argument in the proof of Lemma \ref{lem:Evv}.  However, such a result will not be used in this paper.
In fact, we only need the following very simple estimate.
  \begin{lemma}\label{lem:Evp}
  	Let $T>0$. Suppose $f\in C_c^+  (\mR^d)$
  	and $v$ is the unique solution to  \eqref{eq:PAM}. Then for any $p\geq 2$,
  there is a constant $C = C(d,p,T,\|\cC\|_\infty) > 0$ so that
  	\begin{equation}\label{eq:E-vp}
	  \sup_{x\in \mR^d, t\in[0, T]} \bE [ |v (t, x) |^p ] \leq C \|f\|_\infty^p.
  	\end{equation}
  \end{lemma}
  \begin{proof}
    Recall that
  	$$
  	v(t,x)=P_tf(x)+\int_0^t\!\!\int_{\mathbb{R}^d}p(t-s,x-z)v(s,z)W(\d s,z)\d z.
  	$$
  	Set
  	$$
  	Y(t,x; r):=\int_0^r\!\!\int_{\mR^d} p(t-s,x-y)v(s,y)W(\d s,y)\d y, \quad r\in [0,t],
  	$$
  	and $Y(t,x):=Y(t,x;t)$. Noting that $r\mapsto Y(t,x; r)$ is a continuous martingale, by Burkholder-Davis-Gundy inequality,
  Jensen's inequality and the boundedness of $\cC$, there is a constant
$C =C(d,p)$, whose exact value may change from line to line, so that
  	\begin{equation*}
  		\begin{aligned}
  			&h(t,x):=\bE \left[  | v(t,x) |^p \right]\leq C (P_t f)^p(x) + C \bE\left[ \<Y(t,x;\cdot)\>_t^{p/2} \right] \\
  			\leq& C \|f\|_\infty^p+C \bE  \left[ \left( \int_0^t \d s\!\! \int_{\mathbb{R}^{2d}}p(t-s,x-y)p(t-s,x-z) v(s,y) \cC (y,z) v(s,z) \d y\d z \right)^{p/2} \right]\\
  			\leq& C \|f\|_\infty^p  +C t^{(p-2)/2}\|\cC\|^{p/2}_\infty \int_0^t \d s \int_{\mR^{2d}} p(t-s,x-y)p(t-s,x-z)  [h(s, y)+ h(s, z) ] \d y \d z \\
  			\leq & C \|f\|_\infty^p+ C t^{(p-2)/2}\|\cC\|^{p/2}_\infty \int_0^t  \left[p_{t-s}*h(s)\right](x) \d s.
  		\end{aligned}
  	\end{equation*}
  	Since $f\in C_c^+ (\R^d)$, by  Lemma \ref{lem:spde}(ii), $h\in L^\infty([0,T]; L^q(\mR^d))$
	 for any $q\in [1,\infty)$. Thus,
	 $$
	  \Big \| \int_0^t [p_{t-s}*h(s)] \d s \Big\|_{\infty}
	  \leq C\|h\|_{L^\infty([0,T]; L^d(\mR^d))} \int_0^t s^{-1/2}\d s
	 =2 C t^{1/2} \|h\|_{L^\infty([0,T]; L^d(\mR^d))} .
	 $$
		 Therefore,
  	$$
  	  \|h(t)\|_\infty \leq C \|f\|_\infty^p + C t^{(p-1)/2}\|\cC\|^{p/2}_\infty \int_0^t \| h(s) \|_\infty \d s,
  	$$
  	which  together with Gronwall's inequality yields \eqref{eq:E-vp}.
  \end{proof}

  Next we study the continuity and growth rate of the solution $v$ to the linear SPDE \eqref{eq:PAM}.
  \begin{lemma}\label{lem:control-v}
  	Let $T>0$, $\gamma\in (0, 1)$. Suppose $f\in C_c^+  (\mR^d) $
  	and $v$ is the unique solution of \eqref{eq:PAM}. Then for any $p\in\mN$,
	 there exists constant $C= C(d, p,\gamma, T, \|\cC\|_\infty, \|f\|_\infty)>0$ so that
	 for any $t\in [0,T]$ and \(x, \bar x\in \mR^d\),
  	\begin{equation}\label{eq:v-holder}
		\begin{aligned}
			\bE \left[ \big|\l[v(t,x)-P_tf(x)\r]-\l[v(t,\bar x)-P_tf(\bar{x})\r]\big|^{2p} \right] \leq C |x-\bar{x}|^{2\gamma p}.
		\end{aligned}
  	\end{equation}
  	Consequently, for any $\rho>0$,
  	\begin{equation}\label{eq:v-growth}
  		\bE  \left[ \left( \sup_{x\in \mR^d} |v(t,x) |
		\phi_\rho(x)
		\right)^{2p}\right] \leq C(d, p,\rho,T, \|\cC\|_\infty, \|f\|_\infty) <\infty  \quad  \hbox{for }  t\in [0,T],
  	\end{equation}
  where $\phi_\rho$ is defined by \eqref{def-phi}.
  \end{lemma}
  \begin{proof}
  	Noting that for $0<\gamma<1$, there exists $C(\gamma)>0$ such that
  	\begin{equation}\label{eq:df-p}
  		|p(t,x)-p(t,y)|\leq C |x-y|^\gamma t^{-\gamma/2} (p(t,x)+p(t,y)).
  	\end{equation}
	(see \cite{chen2017heat} for instance).
  	Let $Y(t,x; r)$ and $Y(t,x)$ be the same functions as in the proof of Lemma \ref{lem:Evp}. By Burkholder-Davis-Gundy inequality, for any  $x, \bar{x}\in \mR^d$ and $p\in \mN$, we have
  	\begin{equation*}
  		\begin{aligned}
			&\bE \left[ \left|\l[v(t,x)-P_tf(x)\r]-\l[v(t,\bar x)-P_tf(\bar{x})\r]\right|^{2p} \right]\\
  			= &\bE  \left[ |Y(t,x)-Y(t,\bar{x})|^{2p} \right] \leq C(p) \bE \left[ \<Y(t,x;\cdot)-Y(t,\bar x;\cdot)\>_t^p \right]\\
  			\leq &C \bE   \left[ \l\<\int_0^{\cdot}\int_{\mR^d} p(t-s,x-y)-p(t-s,\bar{x}-y)v(s,y)W(\d s,y)\d y\r\>_t^{p} \right] \\
  			\leq &C \bE  \left[ \Big|\int_0^t \d s\int_{\mR^{2d}}[p(t-s,x-y)-p(t-s,\bar{x}-y)][p(t-s,x-z)-p(t-s,\bar{x}-z)] \right.\\
  			&  \hskip 0.8truein  \left. \times v(s,y)v(s,z)\cC (y,z)\d y\d z \Big|^{p}\right] \\
  			\overset{\eqref{eq:df-p}}{\leq} & C  |x-\bar{x}|^{2\gamma p} \bE   \left[\Bigg( \int_0^t |t-s|^{-\gamma} \d s\int_{\mR^{2d}} \l[ p(t-s,x-y)+p(t-s,\bar{x}-y)\r] \right.\\
  			&   \hskip 0.8truein  \left. \cdot \l[ p(t-s,x-z)+p(t-s,\bar{x}-z)\r]
			 | v(s,y)v(s,z)|  \d y\d z \Bigg)^{p} \right]\\
			\leq & C\, |x-\bar{x}|^{2\gamma p} \int_{[0,t]^p} \prod_{i=1}^p |t-s_i|^{-\gamma} \, \mathrm{d} s_1 \cdots \mathrm{d} s_p\times\\
			&\int_{\mR^{pd}} \int_{\mR^{pd}} \prod_{i=1}^p \Big\{\l[p(t-s_i,x-y_i)+p(t-s_i,\bar{x}-y_i)\r]\cdot \l[ p(t-s_i,x-z_i)+p(t-s_i,\bar{x}-z_i)\r] \Big\}\\
  			& \hskip 0.5in \bE  \prod_{i=1}^p  |  v(s_i,y_i)v(s_i,z_i)  |   \d y_1 \cdots \d y_p ~ \d z_1 \cdots \d z_p .
  		\end{aligned}
  	\end{equation*}
  	On the other hand, by \eqref{eq:E-vp} and H\"older's inequality, we have
  	\begin{equation*}
  		\begin{split}
  			\bE  \prod_{i=1}^p  |  v(s_i,y_i)v(s_i,z_i)  |
			\leq &  \prod_{i=1}^p  \left( \bE\left[ | v (s_i,y_i) |^{2p}\right]\right)^{1/(2p)}
			\left(\bE \left[ | v(s_i,z_i)|^{2p} \right]\right)^{1/(2p)}
			\overset{\eqref{eq:E-vp}}{\leq} \|f\|_\infty^{2p}.
  		\end{split}
  	\end{equation*}
    	The above two estimates imply that
  	\begin{equation}\label{eq:Yholder}
  		\begin{aligned}
  			&\bE  \left[ |Y(t,x)-Y(t,\bar{x})|^{2p} \right] \leq C( |x-\bar{x}|^{2\gamma p} \l(\int_0^t |t-s|^{-\gamma}\d s\r)^p
  			  			\leq C |x-\bar{x}|^{2\gamma p},
  		\end{aligned}
  	\end{equation}
	which gives \eqref{eq:v-holder}.

  	\smallskip

  	To prove \eqref{eq:v-growth}, we define
  	$$
  	  M_\alpha^k:=\sup_{x\not=y \hbox{ \tiny in } B_{2^k}(0)} \frac{|Y(t,x)-Y(t,y)|}{|x-y|^\alpha},
  	$$
  	where $0<\alpha<\gamma <1/2$. Put $\tilde x=2^{-k} x$ and $\widetilde Y(\tilde x)=\frac{Y(t, x)}{2^{2k\gamma}}=\frac{Y(t, 2^k\tilde x)}{2^{2k\gamma}}$. By \eqref{eq:Yholder}, for any $\tilde x, \tilde y\in B_{1}(0)$,  it holds that for any $k\geq 0$,
  	\begin{equation}\label{eq:tYholder}
  		\bE \left[ |\widetilde Y(\tilde x)-\widetilde Y(\tilde y)|^{2p} \right]\leq C_1 2^{-4k \gamma p}|x-y|^{2\gamma p}2^{2k\gamma p}\leq C_1|\tilde x-\tilde y|^{2\gamma p}.
  	\end{equation}
  	We may choose $p$ large enough such that $d/(2p)<\gamma$. By Garsia-Rodemich-Rumsey inequality, for each $\alpha\in(0, \gamma -d/(2p))$,
  	\begin{align*}
  		&\bE  \left[ \left( \sup_{\tilde{x}, \tilde{y}\in B_1(0)}\frac{|\widetilde Y(\tilde x)-\widetilde Y(\tilde y)|}{|\tilde x-\tilde y|^\alpha}\right)^{2p} \right] \leq C \bE \iint_{B_1(0)\times B_1(0)} \frac{|\widetilde{Y}(\tilde x)-\widetilde{Y}(\tilde y)|^{2p}}{|\tilde{x}-\tilde{y}|^{2d+2\alpha p}}\d \tilde{x}\d \tilde{y}\\
  		\overset{\eqref{eq:tYholder}}{\leq}& C\iint_{B_1(0)\times B_1(0)} |\tilde{x}-\tilde{y}|^{2(\gamma-\alpha)p-2d}
\d \tilde{x} \d \tilde{y}
  \leq C \int_0^2 r^{2(\gamma-\alpha)p-d-1} \d r\leq C,
  	\end{align*}
  	which implies that
  	\begin{equation}\label{eq:mean-M}
  		\bE \left[ \left( M^k_\alpha \right)^{2p} \right] \leq C 2^{2(2\gamma-\alpha)kp}.
  	\end{equation}
  	This yields that
  	\begin{align*}
  		&\sup_{x\in \mR^d}\l|Y(t,x) \phi_\rho(x)\r|\leq |Y(t,0)|+\sup_{x\in \mR^d} \l|Y(t,x)-Y(t,0)\r|\phi_\rho(x)\\
  		\leq& |Y(t,0)|+\sum_{k=0}^\infty \sup_{x\in B_{2^k}(0)} M^k_\alpha |x|^{\alpha}\phi_\rho(x)\leq |Y(t,0)|+\sum_{k=0} M^k_\alpha 2^{k(\alpha-\rho)}.
  	\end{align*}
Combing this with \eqref{eq:mean-M}, using  the triangle inequality for the $L^{2p}$-norm,
and choosing \(p, \alpha, \gamma\) such that $d/(2p)<\alpha+d/(2p)<\gamma<1/2 \wedge \rho/2$,    we have
  	\begin{align*}
  		&\l[ \bE \l(\sup_{x\in\mR^d} |Y(t,x)\phi_\rho(x)|^{2p}\r) \r]^{1/(2p)} \\
		\leq&  \l[ \bE \left( |Y(t,0)|^{2p} \right) \r]^{1/(2p)}+  \sum_{k=1}^\infty \left\{ \bE \left[ \left( M^k_\alpha \right)^{2p} \right] \right\}^{1/(2p)}  2^{k(\alpha-\rho)}\\
		\leq &  \l[ \bE \left( |Y(t,0)|^{2p} \right) \r]^{1/(2p)}+C \sum_{k=0}^\infty 2^{(2\gamma-\rho)k}\\
  		\leq  C+ \l[ \bE \left(|Y(t,0)|^{2p}\right) \r]^{1/(2p)}.
  	\end{align*}
  	Then estimate \eqref{eq:v-growth} follows by using Lemma \ref{lem:Evp} and noting that $v(t,x)=P_tf(x)+Y(t,x)$
    with $P_tf $ being bounded.
  \end{proof}

  \begin{remark}\label{rmk:initial-bdd}
       Leveraging Lemma \ref{lem:comparion}, Remark \ref{rmk:bdd-initial} and the monotone convergence theorem, it becomes evident that all the conclusions in Lemma \ref{lem:Evv}, \ref{lem:Evp} and \ref{lem:control-v} hold for all $f\in C_b(\mR^d)$.
  \end{remark}

  Recall that $\phi_\rho(x)=(1+|x|^2)^{- \rho/2}$. When $0\leq f\leq C\phi_\rho$    for some
  $\rho>0$, we can refine the estimate \eqref{eq:v-growth} further. To do this, the following auxiliary result is needed.
    \begin{lemma}
  	  For any $\rho, T>0$, there is a constant $C=C(d ,\rho,T)$ such that
   \begin{equation}\label{eq:Ptphi}
  		P_t\phi_{\rho} \leq C \phi_{\rho},\quad \forall t\in [0,T].
  	\end{equation}
  \end{lemma}
  \begin{proof}
  	For any $t\in[0,T]$ and $|x|> 1\vee T$,
  	\begin{align*}
  		P_t\phi_{\rho}(x)&=(2\pi t)^{-d/2}\int_{\mR^d} \phi_{\rho}(x-y)\e^{-\frac{|y|^2}{2t}}\d y\\
  		&=(2\pi )^{-d/2}\int_{|z|\leq |x|}\phi_{\rho}(x-\sqrt{t}z)\e^{-|z|^2/2}\d z+(2\pi )^{-d/2}\int_{|z|> |x|}\phi_{\rho}(x-\sqrt{t}z)\e^{-|z|^2/2}\d z\\
  		&\leq C\phi_{\rho}(x)\int_{|z|\leq |x|}\e^{-|z|^2/2}\d z+C\int_{|z|> |x|}\e^{-|z|^2/ 2}\d z\\
  		&\leq C\phi_{\rho}(x)+C\int_{|x|}^\infty \e^{-r}\d r\leq C\phi_{\rho}(x).
  	\end{align*}
  	So \eqref{eq:Ptphi} holds for all $t\in [0,T]$ and $x\in \mR^d$, since $P_t \phi_\rho(x) \leq C$.
  \end{proof}

  The following result is an improvement of \eqref{eq:v-growth} when
  $f$ is dominated by some $\phi_\rho$.
  \begin{lemma}
  	Let $\rho, T>0$ and $p\in \mN$.
  If $0\leq f\leq N \phi_{\rho}$  for some constant $N>0$, then there is a constant $C=C(d, p, \rho, T, N, \|\cC\|_\infty)$ such that
  	\begin{equation}\label{eq:v-decay}
  		\bE  \left[ \left( \sup_{x\in \mR^d}  |v(t,x)  \phi^{-1}_{\rho/2}(x)|\right)^{2p} \right] \leq C(d, p,\rho,T, N, \|\cC\|_\infty)<\infty,  \quad  \forall  t\in [0, T].
  	\end{equation}
  \end{lemma}
  \begin{proof}
	Let \(w:=v\phi_{\rho}^{-1}\). Then \(w\) satisfies the following SPDE:
	\[
	\partial_t w = \frac{1}{2}\Delta w + \nabla w \cdot \frac{\nabla \phi_{\rho}}{\phi_{\rho}} + w \left( \frac{1}{2} \frac{\Delta \phi_{\rho}}{\phi_{\rho}} + \partial_t W \right), \quad w(0,x) = f(x) \phi_{\rho}^{-1}(x)\in C_b(\mR^d).
	\]
    Let \(q(t,x,y)\) be the heat kernel associated to the parabolic operator \(\partial_t - \frac{1}{2}\Delta - \frac{\nabla \phi_{\rho}}{\phi_{\rho}} \cdot \nabla\). Since \(b:=\frac{\nabla \phi_{\rho}}{\phi_{\rho}}\) and \(c:=\frac{\Delta \phi_{\rho}}{\phi_{\rho}}\) are smooth and bounded, thanks to \cite[Theorem 1.1]{chen2017heat}, for each \(t\in (0,T]\) and \(x,y\in \mR^d\), we have
	\[
	q(t,x,y)\leq C t^{-\frac{d}{2}} \e^{-\frac{|x-y|^2}{Ct}}\leq C p(Ct,x-y),
	\]
	for some \(C=C(d,T)>0\), and for each \(\gamma\in (0,1)\), there exists a constant \(C= C(d,\gamma, T)>0\) such that
	\[
	|q(t,x,y)-q(t,\bar{x},y)|\leq C |x-\bar{x}|^\gamma t^{-\gamma/2} (p(Ct,x-y)+p(Ct,\bar{x}-y)).
	\]
	Noting that
	\[
	w(t,x)=\int_{\mR^d} q(t,x,y) f(y)\phi_{\rho}^{-1}(y) \d y + \int_0^t\!\!\int_{\mR^d} q(t-s,x,y) w(s,y) [c(y) \d s+ W(\d s,y)]\d y.
	\]
	Then following the same argument as in Lemma \ref{lem:control-v}, we can see that
	\[
	\bE  \left[\left( \sup_{x\in \mR^d}  |v(t,x)  \phi^{-1}_{\rho/2}(x)|\right)^{2p}\right]=\bE  \left[\left| \sup_{x\in \mR^d} w(t,x)  \phi_{\rho/2}(x) \right|^{2p}\right]\leq C(d, p,\rho,T, N, \|\cC\|_\infty).
	\]
	So we obtain our desired assertion.
  \end{proof}

  \subsection{Derivative of the conditional log-Laplace}
  Now let $u(\lambda)$ be the solution to \eqref{eq:spde-clt} with initial date $u(\lambda; 0)=\lambda f\in C^+_c(\mR^d)$ and $v$ be the solution to \eqref{eq:PAM}. Below, we investigate the derivative of \(u(\lambda)\) with respect to \(\lambda\). As we will show in the next section, this derivative is crucial for deriving the moment formulas.

  \smallskip

  By Lemma \ref{lem:comparion},  $u(\lambda)\leq \lambda v$ and $u(\lambda)\leq u(\lambda+\delta)$, for any $\delta, \lambda\geq 0$. Put
  \begin{equation}\label{eq:def-wdelta}
  	 w_\delta(\lambda; t,x):=\frac{u(\lambda+\delta; t,x)-u(\lambda;t,x)}{\delta}, \quad \forall \delta>0
  \end{equation}
  Then $w_\delta(\lambda; t,x)\geq 0$ solves equation
  \begin{equation}\label{eq:wdelta}
  	\begin{cases}
  		\partial_t w_\delta(\lambda)=\frac{1}{2}\Delta w_\delta(\lambda)+w_\delta(\lambda){\p_t W}-k_{\delta}(\lambda),\\
  		w_\delta(\lambda,0)=f,
  	\end{cases}
  \end{equation}
  where
  $$
    k_{\delta}(\lambda):=\frac{1}{2\delta}\l(u^2(\lambda+\delta)-u^2(\lambda) \r) \geq 0.
  $$
  Again by Lemma \ref{lem:comparion},
  \begin{equation}\label{eq:w-less-v}
  	0\leq w_\delta(\lambda)\leq v,
  \end{equation}
  which yields
  $$
    \|u(\lambda+\delta)-u(\lambda)\|_{L^2([0,T]\times \mR^d\times\Omega)} \leq \delta \|v\|_{L^2([0,T]\times \mR^d\times\Omega)} \leq C(T) \delta, \quad \forall T>0.
  $$
  Hence, $u(\lambda): \mR_+\rightarrow L^2([0,T]\times\mathbb{R}^d\times\Omega)$ is a Lipschitz function. Therefore, $u(\lambda)$ has derivative in $L^2([0,T]\times\mathbb{R}^d\times\Omega)$, say $\p_\lambda u$, such that for a.e. $\lambda\in \mR_+$,
  \begin{equation}\label{eq:converge-u'}
  	\|w_\delta(\lambda)-\p_\lambda u(\lambda)\|_{L^2([0,T]\times \mR^d\times\Omega)}\rightarrow 0  \mbox{ as }\delta\to 0,
  \end{equation}
  and for any $\lambda>0$,
  $$u(\lambda;\cdot,\cdot)=\int_0^\lambda \p_\lambda u(\tau; \cdot,\cdot) \, \d\tau \quad
  \mbox{in } L^2([0,T]\times \mathbb{R}^d), \quad\mbox{a.s.}
  $$
  By the same argument,
    one gets that for any $\lambda\geq 0$,
  \begin{equation}\label{eq:wtoDu-L4}
	w_\delta(\lambda) \hbox{ converges to }    \p_\lambda u(\lambda)  \hbox{ in } L^4([0,T]\times\mR^d\times\Omega) \hbox{ as } \delta \to 0.
  \end{equation}

  Next we derive the equation for $\p_\lambda u$. For each fixed $\lambda\geq 0$, consider the following linear SPDE with respect to $w(\lambda; \cdot,\cdot)$:
  \begin{equation}\label{eq:u'}
  	\begin{cases}
  		\partial_t w(\lambda)=\frac{\Delta}{2}w(\lambda)+w(\lambda)\p_t W-u(\lambda)\p_\lambda u(\lambda)
  		,\\
  		w(\lambda,0,x)=f(x).
  	\end{cases}
  \end{equation}
  Noting that $k_{\delta}(\lambda; t,x)=\frac{1}{2\delta}(u^2(\lambda+\delta; t, x)-u^2(\lambda; t, x))=\frac{1}{2}(u(\lambda+\delta; t, x)+u(\lambda; t, x))w_\delta(\lambda; t,x)$, we have
  \begin{eqnarray*}
  	&&\bE \int_0^T\!\!\int_{\mR^d}|k_\delta(\lambda; t,x)-u(\lambda;t,x)\p_\lambda u(\lambda;t,x)|^2\d x \d t\\
  	&\leq&
  	\bE\int_0^T\!\! \int_{\mR^d}|u(\lambda;t,x)|^2|w_\delta(\lambda; t,x)-\p_\lambda u(\lambda;t,x)|^2\d x \d t\\
  	&&+\frac{1}{4}\bE\int_0^T\!\!\int_{\mR^d} |(u(\lambda+\delta; t,x)-u(\lambda;t,x))|^2||w_\delta(\lambda; t,x)|^2\d x \d t\\
  	&\leq&
  	\left(\bE \int_0^T\!\!\int_{\mR^d}|u(\lambda;t,x)|^4\d x \d t\right)^{1/2}
	\left( \bE \int_0^T\!\!\int_{\mR^d}|w_\delta(\lambda; t,x)-\p_\lambda u(\lambda;t,x)|^4\d x \d t\right)^{1/2}\\
  	&& +\frac{1}{4}\left( \bE \int_0^T\!\!\int_{\mR^d}|(u(\lambda+\delta; t,x)-u(\lambda;t,x))|^4\d x \d t\right)^{1/2}
	\left( \bE \int_0^T\!\!\int_{\mR^d}|w_\delta(\lambda; t,x)|^4\d x \d t\right)^{1/2}\\
  	&\overset{\eqref{eq:wtoDu-L4}}{\longrightarrow}  & 0 \quad \mbox{as }\delta\rightarrow 0.
  \end{eqnarray*}
  This together with \eqref{eq:wdelta}, \eqref{eq:u'} and Lemma \ref{lem:spde} (ii) yields
  \begin{equation*}
  	\bE \left[ \sup_{0\leq t\leq T}\|w_\delta(\lambda; t)-w(\lambda; t)\|_2^2\right] +\bE \int_0^T\|w_\delta(\lambda; t)-w(\lambda; t)\|_{H^1}^2 \d t\rightarrow 0
  	,\quad \mbox{as }\delta\to 0.
  \end{equation*}
 Here $H^1$ denotes the Sobolev space
  $\{f\in L^2(\R^d): \nabla f \in L^2(\R^d)\}$ of order $(1, 2)$ equipped with the Hilbert norm
  $\| u\|_{H^1}:= \big(\int_{R^d} (|f (x) |^2 + |\nabla f (x) |^2 ) dx \big)^{1/2}$.
  This together with \eqref{eq:converge-u'} yields  $w(\lambda)=\p_\lambda u(\lambda)$ in $L^2([0,T]\times\mathbb{R}^d\times\Omega)$.
  Thus, $\p_\lambda u(\lambda)$ solves \eqref{eq:u'}, i.e.,
  $\p_\lambda u(\lambda)\geq 0$ satisfies
  \begin{equation}\label{eq:Eq-Du}
  	\begin{cases}
  		\partial_t [\p_\lambda u(\lambda)]=\frac{\Delta}{2}\p_\lambda u(\lambda)+\p_\lambda u(\lambda)\p_t W-u(\lambda)\p_\lambda u(\lambda),\\
  		\p_\lambda u(\lambda;0)=f.
  	\end{cases}
  \end{equation}
  Again by Lemma \ref{lem:comparion}, we have $0\leq \p_\lambda u(\lambda)\leq v$.

 \smallskip

  To sum up, we have the following.

  \begin{proposition}
    Let $f\in C_c^+(\mR^d)$. Then $0\leq u(\lambda)\leq \lambda v$, and there is a random function $\p_\lambda u: \overline{\mR}_+\times \overline{\mR}_+ \times \mR^d\times \Omega\to \mR$ such that
     \begin{equation}\label{eq:u'-less-v}
    	0\leq \p_\lambda u(\lambda) \leq v ~\mbox{ and }~   u(\lambda)= \int_0^\lambda \p_\lambda u(\tau) \d \tau.
    \end{equation}
    Moreover, $\p_\lambda u(\lambda)$ solves equation \eqref{eq:Eq-Du}.

  \end{proposition}

  \section{Moment formulas for SBMRE}\label{Sec:MF}
  The moment formulas are very important in  studying the properties of classical superprocesses, for examples,  the existence of density, law of large number, local extinction and local exponential growth (see \cite{englander2007branching} and the reference therein). In this section, we employ the conditional Laplace transform to derive foundational formulas for the first two moments of SBMRE, which are also of great significance to our study. A similar methodology was previously employed in the study by Xiong \cite{Xiong2004SLL} concerning superprocesses over a stochastic flow.

  \smallskip

  Recall that $B_t$ and $B_t'$ are two independent $d$-dimensional standard Brownian motion starting respectively from
  $x,y\in\R^d$ under $P_{(x,y)}$. Let $Q_t$ be the transition semigroup of standard Brownian motion $(B_t, B'_t)$ on $\R^{2d}$.
  For any $F\in{\cB}(\mR^{2d})$, define
  $$
  Q^{\cC}_t F(x,y):=
  E_{(x,y)}\l[F(B_t,B'_t)\exp\left(\int_0^t \cC(B_s,B'_s)\d s\right)\r],\quad \forall x,y\in\mR^d.
  $$
  Let $\pi$ be a map from ${\cB}(\mR^{2d})$ to ${\cB}(\mR^{d})$ defined by
  $\pi F(x):=F(x,x)$, $\forall x\in\mR^d$.

  \begin{proposition}[Moment Formulas]\label{prop:moments}
  	Let $\nu\in \cM_F(\mR^d)$ and $f\in C_b(\mR^d)$. Then
  	\begin{equation}\label{eq:first}
  		\bE_\nu\langle f, X_t \rangle= \langle P_tf, \nu\rangle
  	\end{equation}
  	and
  	\begin{equation}\label{eq:second}
  		\bE_\nu \left[ \langle f, X_t \rangle^2\right] =\langle Q_t^{\cC} (f\otimes f), \nu\otimes\nu\rangle+\l\langle\int_0^t P_{t-s}(\pi Q_s^{\cC} (f\otimes f))\d s,\nu\r\rangle.
  	\end{equation}
  	
  \end{proposition}

  Recall that $v$ is the unique solution of (\ref{eq:PAM}), $u(\lambda)$ is the solution to \eqref{eq:clt} with initial data $\lambda f$ and $w_\delta$ is defined in \eqref{eq:def-wdelta}.  For any $\nu\in\cM_F(\mR^d)$, the next lemma shows that $w_{\delta}(\lambda; t)\to \p_\lambda u(\lambda; t)$ and $\p_\lambda u(\delta; t)\to v(t)$ in $L^1(\nu\times \bP)$, as $\delta\to0$.
  \begin{lemma}\label{lem:limit-u'}
  	Let $f\in C_c^+(\mR^d)$, and $\nu\in\cM_F(\mR^d)$. Then for each $t\geq 0$ and $\lambda\geq 0$,
  	\begin{equation}\label{eq:limit-u'}
  		\lim_{\delta\to 0}\bE \int_{\mR^d}|w_\delta(\lambda; t,x)-  \partial_\lambda u (\lambda; t,x)|\nu(\d x)=0,
    \end{equation}
  	and
  	\begin{equation}\label{eq:Du-v}
  		\lim_{\delta\rightarrow 0}\bE \int|v(t,x)-  \partial_\lambda u (\delta; t,x)|^2\nu(\d x)= 0.
  	\end{equation}
  \end{lemma}
  \begin{proof}
  	
  	Note that
  	\begin{align*}
  		w_\delta(\lambda; t,x)=&P_tf(x)+\int^t_0\int_{\mR^d}p(t-s, x-z)w_\delta(\lambda;s,z)W(\d s, z)\d z\\
  		&-\int^t_0P_{t-s}\left( \frac{u(\lambda;s)+u(\lambda+\delta;s)}{2}w_\delta(\lambda;s)\right)(x)\d s,
  	\end{align*}
  	and
  	\begin{align*}
  		\partial_\lambda u
  		(\lambda; t,x)=&P_tf(x)+\int^t_0\int_{\mR^d}p(t-s, x-z)
  		\partial_\lambda u
  		(\lambda;s,z)W(\d s, z)\d z\\
  		&-\int^t_0P_{t-s}(u(\lambda; s)
  		\partial_\lambda u
  		(\lambda;s))(x)\d s.
  	\end{align*}
  	Put
  	$$
  	  I_1(x):=\left|\int^t_0P_{t-s}\left( \frac{u(\lambda;s)+u(\lambda+\delta;s)}{2}w_\delta(\lambda;s)-u(\lambda;s)\partial_\lambda u(\lambda;s)\right)\d s\right|
  	$$
  	and
  	$$
  	  I_2(x):=\left|\int^t_0\!\!\int_{\mR^d}p(t-s, x,z)[w_\delta(\lambda;s,z)-
  	  \partial_\lambda u(\lambda;s,z)]W(\d s, z)\d z\right|
    $$
  	Then,
  	$$
  	  |w_\delta(\lambda; t,x)- \partial_\lambda u (\lambda; t,x)|\leq I_1(x)+I_2(x).
  	$$
  	Using the facts that $w_\delta(\lambda)\leq v$ and $u(\lambda)\leq \lambda v$,  we have
  	\begin{equation*}
  	  \begin{aligned}
  	  	&\bE \int_{\mR^d}I_1(x)\,\,\nu(\d x)\\
  	  	\leq &\bE \int_{\mR^d}\nu(\d x)\int_0^t \d s\int_{\mR^d}p(t-s,x-y)u(\lambda;s,y)| \partial_\lambda u  (\lambda;s,y)-w_\delta(\lambda;s,y)|\d y\\
  	  	&+\frac{\delta}{2}\bE \int_{\mR^d}\nu(\d x)\int_0^t \d s\int_{\mR^d}p(t-s,x-y)w^2_\delta(\lambda;s,y)\d y\\
  	  	\leq & \lambda \int_{\mR^d}\nu(\d x) \int_0^t \d s\int_{\mR^d}p(t-s,x-y)\bE \left( v(s,y)|\partial_\lambda u  (\lambda;s,y)-w_\delta(\lambda;s,y)|\right)\d y\\
  	  	&+\frac{\delta}{2} \int_0^t \langle P_{t-s}(\bE [v^2_s]),\nu\rangle \d s.
  	  \end{aligned}
  	\end{equation*}
  	Thanks to \eqref{eq:converge-u'}, for a.e  $(s,y)\in \mR_+\times\mR^d$,
  	$$
  	  \bE \left[ v(s,y)|{\partial_\lambda u} (\lambda;s,y)-w_\delta(\lambda;s,y)|\right]\rightarrow 0 \quad\mbox{as }\delta\to 0,
  	$$
  	by the dominated convergence theorem, we get
  	$$
  	  \lim_{\delta\to 0}\int_{\mR^d}\nu(\d x) \int_0^t \d s\int_{\mR^d}p(t-s,x-y)\bE  \l[ v(s,y)|{  \partial_\lambda u} (\lambda;s,y)-w_\delta(\lambda;s,y)| \r] \d y= 0.
  	$$
  	Therefore, $\lim_{\delta\to 0}\bE \int_{\mR^d}I_1(x)\,\,\nu(\d x)=0$. For $I_2(x)$, using Burholder-Davis-Gundy inequality and the same argument above, we also have $\lim_{\delta\to 0}\bE \int_{\mR^d}I_2(x)\,\,\nu(\d x)=0$.  So we finish the proof of \eqref{eq:limit-u'}.
  	
  	\smallskip
  	
  	Equation \eqref{eq:Du-v} can be obtained by using the similar arguments, so we omit its details here.
  \end{proof}

  \begin{proof}[Proof of Proposition \ref{prop:moments}]
    By the monotone convergence theorem, we only need to consider the case $f\in C^+_c(\mR^d)$. According to the conditional Laplace functional formula \eqref{eq:clt}, we have
  	\begin{equation}\label{eq:dlaplace}
  		\begin{aligned}
  			&-\frac{1}{\delta}\bE_\nu \left(\e^{-(\lambda+\delta)\langle f, X_t \rangle}-\e^{-\lambda\langle f, X_t \rangle}\right)\\
  			=&-\frac{1}{\delta}\bE \left(\e^{-\langle u(\lambda+\delta;t),\nu\rangle}-\e^{-\langle u(\lambda;t),\nu\rangle}\right)=\bE \left(\e^{-\langle u(\lambda;t),\nu \rangle}\frac{1-\e^{-\langle u(\lambda+\delta;t)-u(\lambda;t),\nu\rangle}}{\delta}\right).
  		\end{aligned}
  	\end{equation}
  	This together with \eqref{eq:limit-u'} and \eqref{eq:E-vp} implies
  	$$
  	  \lim_{\delta\to 0+}-\frac{1}{\delta}\bE_\nu \left(\e^{-(\lambda+\delta)\langle f, X_t \rangle}-\e^{-\lambda\langle f, X_t \rangle}\right)=\bE (\e^{-\langle u(\lambda;t),\nu \rangle}\< \p_\lambda u(\lambda;t),\nu\>)\leq \bE\<v(t), \nu\><\infty.
  	$$
  	In particular, for $\lambda=0$, by Fatou's lemma, we get
  	$$\bE_\nu\langle f, X_t \rangle=\bE_\nu\lim_{\delta\rightarrow 0^{+}}\frac{1}{\delta} \left(1-\e^{-\delta\langle f, X_t \rangle}\right)\leq\lim_{\delta\rightarrow 0^{+}}\frac{1}{\delta}\bE_\nu \left(1-\e^{-\delta\langle f, X_t \rangle}\right)<\infty,$$
  	Letting $\delta\to0$ in \eqref{eq:dlaplace}, by
  	the dominated convergence theorem,
  	\begin{equation}\label{eq:DLaplace}
  		\bE_\nu\left(\langle f, X_t \rangle \e^{-\lambda\langle f, X_t \rangle}\right)=\bE \left(\e^{-\langle u(\lambda;t),\nu\rangle}\langle \p_\lambda u(\lambda;t),\nu\rangle\right),\quad \forall \lambda\ge 0.
    \end{equation}
  	Letting $\lambda= 0$, we have
  	\begin{equation*}
  		\bE_\nu\langle f, X_t \rangle=\bE \langle \p_\lambda u(0, t), \nu\rangle=\bE \langle v(t), \nu\rangle=\langle P_tf, \nu\rangle.
  	\end{equation*}
  	
  	For the second order moment formula, by \eqref{eq:Du-v} and  \eqref{eq:DLaplace},
  	\begin{equation}\label{eq:D2u}
  		\begin{aligned}
  			&\lim_{\delta\rightarrow 0^{+}}\frac{1}{\delta}\bE_\nu\langle f, X_t \rangle\left(1-\e^{-\delta\langle f, X_t \rangle}\right)\\
  			\overset{\eqref{eq:DLaplace}}{=}&\lim_{\delta\rightarrow 0^{+}} \bE \left[\frac{-\e^{-\langle u(\delta;t),\nu\rangle}\langle \p_\lambda u(\delta;t),\nu\rangle+\langle \p_\lambda u(0;t),\nu\rangle}{\delta}\right]  \\
  			=&\lim_{\delta\rightarrow 0^{+}} \bE \left[\frac{\left(1-\e^{-\langle u(\delta;t),\nu\rangle}\right)\langle \p_\lambda u(\delta;t),\nu\rangle}{\delta}\right]-\lim_{\delta\rightarrow 0^{+}} \bE \left\langle\frac{\p_\lambda u(\delta;t)-\p_\lambda u(0;t)}{\delta},\nu\right\rangle\\
  			\overset{\eqref{eq:Du-v}}{=}&\bE [ \langle v(t),\nu\rangle^2] -\lim_{\delta\rightarrow 0^{+}}\bE \left\langle\frac{\p_\lambda u(\delta;t)-\p_\lambda u(0;t)}{\delta},\nu\right\rangle.
  		\end{aligned}
  	\end{equation}
  	By Lemma \ref{lem:Evv}, we have
  	\begin{eqnarray*}
  		\bE [\langle v(t),\nu\rangle^2] &=&\bE \left(\int_{\mR^d} v(t,x)\nu(\d x) \int_{\mR^d} v(t,y)\nu(\d y)\right)\\
  		&=&\int_{\mR^{2d}} \bE (v(t,x)v(t,y))\nu(\d x)\nu(\d y)\\
  		&=& \langle Q_t^{\cC} (f\otimes f), \nu\otimes\nu\rangle.
  	\end{eqnarray*}
  	For the second term on the right-hand side of \eqref{eq:D2u},  set
  	$$
  	  \eta_\delta(t,x):=\frac{\p_\lambda u(\delta; t,x)-\p_\lambda u(0;t,x)}{\delta}, \quad \forall \delta>0.
  	$$
  	Then $\eta_\delta$ satisfies
  	\begin{equation*}
  		\begin{cases}
  			\partial_t \eta_\delta=\frac{1}{2}\Delta \eta_\delta+\eta_\delta \p_t W-k_\delta,\\
  			\eta_\delta(0)=0,
  		\end{cases}
  	\end{equation*}
  	where $k_\delta(t,x)=u(\delta; t,x)\p_\lambda u(\delta; t,x)/\delta$. That is to say
  	\begin{equation*}
		\eta_\delta(t,x)= -\int_0^tP_{t-s}(k(s,\cdot))(x) \d s +\int_0^t\int_{\mathbb{R}^d} p(t-s,x-y)\eta_\delta(s,y)W(\d s,y) \d y.
    \end{equation*}
  	Let $\eta(t,x)$ be the unique continuous solution of the following equation:
  	\begin{equation*}
  		\begin{cases}
  			\partial_t \eta=\frac{1}{2}\Delta \eta+\eta \p_t W-v^2,\\
  			\eta(0)=0.
  		\end{cases}
  	\end{equation*}
  	The mild form of the above equation is
  	 \begin{equation}\label{eq:int-eta}
  		 \eta(t,x)= -\int_0^tP_{t-s}(v^2(s,\cdot))(x) \d s +\int_0^t\int_{\mathbb{R}^d} p(t-s,x-y)\eta(s,y)W(\d s,y)\d y.
    \end{equation}
    We get $\p_\lambda u(\delta; t,x)\le v(t,x)$ by \eqref{eq:u'-less-v},  $u(\delta; t,x)/\delta\le v(t,x)$ by Lemma \ref{lem:comparion}. Therefore, we have $0\leq k_\delta \le v^2$.
    Then following the argument of the  proof of Lemma \ref{lem:limit-u'}, we get that,  for each $t\geq 0$,
  	$$
  	  \lim_{\delta\rightarrow 0}\int_{\mR^d}\nu(\d x)\bE \l|\eta(t,x)-\eta_\delta(t,x)\r|=0.
  	$$
  	Taking expectation on both sides of \eqref{eq:int-eta}, we get
  	$$
  	  \bE \eta(t,x)=-\int_0^t \d s \int p(t-s,x-y)\bE [v^2(s,y)] \d y=-\int_0^t P_{t-s}(\pi Q_s^{\cC} (f\otimes f))(x)\d s.
  	$$
  	Therefore,
  	\begin{align*}
  		-\lim_{\delta\rightarrow 0^{+}}\bE \left\langle\frac{\p_\lambda u(\delta;t)-\p_\lambda u(0;t)}{\delta},\nu\right\rangle=
       &-\lim_{\delta\rightarrow 0^{+}} \bE \<\eta_\delta(t), \nu\>= \bE \<\eta(t), \nu\>\\
  		=& \left\langle\int_0^t P_{t-s}(\pi Q_s^{\cC} (f\otimes f))\d s,\nu\right\rangle.
  	\end{align*}
  	This together with \eqref{eq:D2u} implies the desired result.
  \end{proof}

  \section{Persistence in ``weak" environments}\label{Sec:persistence}

  In this section, we  establish the existence and uniqueness of solutions to equation \eqref{eq:CMP} when $\mu\in{\cM}_\rho(\mR^d)$, and prove
   Theorem \ref{thm:1}.

  \subsection{SBMRE starting from infinite measure}
  Suppose $\mu$ and $\mu_n$ are all in $\mathcal{M}(\mathbb{R}^d)$. We say that  $\mu_n$ converges vaguely to $\mu$, if for any $f\in C_c(\mR^d)$ , $\mu_n(f)$ converges to $\mu(f)$ as $n\to\infty$.
  In the following, we first construct the process in the $\sigma$-finite measure space $\cM_\rho(\mR^d)$.
  Let
  $$
    \Phi_\rho:=\l\{f\in C(\mR^d): |f|\leq N \phi_{\rho}, \mbox{ for some } N>0  \r\}.
  $$
  Suppose $\mu$ and $\mu_n$ are all in $\mathcal{M}_\rho(\mathbb{R}^d)$, we say $\mu_n$ converges to $\mu$ in $\cM_\rho(\mR^d)$, if $\mu_n(f)$ converges to $\mu(f)$ as $n\to\infty$ for all $f\in \Phi_\rho$.

   We consider SBMRE starting from an infinite measure $\mu\in \cM_\rho(\mR^d)$ with $\rho>0$ as the solution of \eqref{eq:CMP}.

    \begin{proposition}\label{prop:CMP}
  	Suppose  that  $\mu\in \cM_\rho(\mR^d)$
  with $\rho>0$.  There  is a unique solution
	      $(\Omega,\mathcal{F},\mathcal{F}_t,\bP,X, W)$ solves the \eqref{eq:CMP}   with initial measure $\mu$, and $[0,\infty)\ni t\mapsto X_t\in \cM_{\rho}(\mR^d)$ is continuous a.s.. Moreover, the moment formulas in Proposition \ref{prop:moments} also hold for $f\in \Phi_\rho$ and \(\mu\in \cM_\rho(\mR^d)\).
  \end{proposition}
  \begin{proof}
  	The proof is similar to that of  \cite[Theorem 2.5]{LMX2009flow}.
  	
  	Let $\mu_0=\mu|_{B_(0)}$. For each $i= 1,2, \cdots$, set $S_i=\{x\in\mR^d: 2^{i-1}\leq|x|< 2^i\}$, $\mu_i=\mu|_{S_i}$.
  	We may construct a common probability space $(\Omega,\mathcal{F},\bP)$, on which $X^i$  and  $W$ are defined such that $X^i$, $i\ge 0$, are independent given $W$, and for each $i\geq 0$, $X^i$ is the solution of \eqref{eq:CMP} with noise $W$ and initial measure $\mu_i$ (see \cite{MX2007local}). Set $X=\sum_{i=0}^\infty X^i$.
  	By the first order moment formula and \eqref{eq:Ptphi},
  	$$
  	  \bE\<\phi_{\rho},X_t\>=\sum_{i=0}^\infty\bE\<\phi_{\rho},X^i_t\> =\sum_{i=0}^\infty \<P_t\phi_{\rho}, \mu_i\>=\<P_t\phi_{\rho}, \mu\>\le C\<\phi_{\rho}, \mu\><\infty.
  	$$
  	Hence, for each $t\geq 0$, $X_t\in \cM_\rho(\mR^d)$ a.s.

  	Next we show that $X_\cdot$ is continuous in $\cM_\rho(\mR^d)$. Assume $f\in C^2(\mR^d)$ satisfying $f\geq0$ and $\nabla^k f \leq C_k \phi_{\rho}$. For any $i\in \mN$, using \eqref{eq:MP} with $X$ and $\mu$ being replaced by $X^i$ and $\mu_i$, and Doob's inequality, we have
  	$$
  	  \bE\l[\sup_{0\leq t\leq T}\<f,X^i_t\>^2\r]\leq C \l(|\<\phi_{\rho},\mu_i\>|^2+ \bE \int_0^T \<\phi_{2\rho}, X_t^i\> \d t+ \bE \int_0^T \<\phi_{\rho},X_t^i\>^2 \d t\r)
  	$$
  	By the  moment formulas in Proposition \ref{prop:moments} and \eqref{eq:Ptphi}, we have
  	\begin{equation*}
  	    \begin{aligned}
  	      \sum_{i=0}^\infty\l(\bE \int_0^T  \<\phi_{2\rho}, X_t^i\> \d t\r)^{1/2}\overset{\eqref{eq:first}}{=}& \sum_{i=0}^\infty\l(\int_0^T  \<P_t \phi_{2\rho}, \mu_i\> \d t\r)^{1/2} \overset{\eqref{eq:Ptphi}}{\leq} C\sum_{i=0}^\infty  \<\phi_{2\rho}, \mu_i\>^{1/2} \\
  	      \leq& C \sum_{i=0}^\infty 2^{-i\rho/2}  \<\phi_\rho,\mu_i\>^{1/2}\leq  C  \sum_{i=0}^\infty 2^{-i\rho}+ C \sum_{i=0}^\infty \<\phi_\rho,\mu_i\>\\
  	      \leq& C(1+ \<\phi_\rho, \mu\>)<\infty,
  	\end{aligned}
  	\end{equation*}
  	and
  	\begin{equation*}
  	    \begin{aligned}
  	    \sum_{i=0}^\infty\l\{  \bE \int_0^T \<\phi_{\rho},X_t^i\>^2 \d t \r\}^{1/2} \overset{\eqref{eq:second}}{\leq} & C \sum_{i=0}^\infty\l\{  \int_0^T  \l( \<P_t\phi_{\rho},\mu_i\>^2+ \l\< \int_0^t P_{t-s}[(P_s \phi_\rho)^2] \d s, \mu_i \r\> \r) \d t \r\}^{1/2} \\
  	    \overset{\eqref{eq:Ptphi}}{\leq} & C  \sum_{i=0}^\infty \l\{ \<\phi_\rho, \mu_i\>^2+ \<\phi_{2\rho}, \mu_i\> \r\}^{1/2} \\
  	    \leq & C \sum_{i=0}^\infty \<\phi_\rho, \mu_i\>+ C\sum_{i=0}^\infty  \<\phi_{2\rho}, \mu_i\>^{1/2}<\infty.
  	    \end{aligned}
  	\end{equation*}
    Therefore,
  	$$
  	  \sum_{i=0}^\infty\bE\l[\sup_{0\leq t\leq T}\<f,X^i_t\>\r] \leq \sum_{i=0}^\infty \l[ \bE\sup_{0\leq t\leq T}\<f,X^i_t\>^2\r] ^{1/2}<\infty.
  	$$
  	So with probability $1$, the following uniform convergence holds:
  	\begin{align*}
  	    \lim_{m\to\infty}\sup_{0\le t\le T}\left|\sum^m_{i=0}\langle f, X^i_t\rangle-\langle f, X_t\rangle\right|=& \lim_{m\to\infty}\sup_{0\leq t\le T}\left|\sum^\infty_{i=m+1}\langle f, X^i_t\rangle\right|\\
  	    \leq& \lim_{m\to\infty} \sum_{i=m+1}^\infty\sup_{0\leq t\leq T}\langle f,X^i_t\rangle=0.
  	\end{align*}
  	Hence, the continuity of $t\mapsto\langle f,X^i_t\rangle $ gives the continuity of $t\mapsto\langle f, X_t\rangle$, which implies the continuity of  $t\mapsto X_{t}$  in $\cM_\rho(\mR^d)$.
  	
  	By a limit argument, $X$ is
  	a solution to \eqref{eq:CMP} with $W$ given above and $X_0=\mu$. Moreover, the moment formulas \eqref{eq:first} and \eqref{eq:second} also hold for $f\in \Phi_\rho$ and $X$, and equation \eqref{eq:clt} also holds when $\mu\in \cM_{\rho}(\mR^d)$. Hence, following \cite[Lemma 2.3]{MX2007local}, $X$ is also a solution to \eqref{eq:MP} with initial measure $\mu$.
  \end{proof}

  \subsection{Proof for Theorem \ref{thm:1}}
  In this subsection, we prove Theorem \ref{thm:1}. Before going to the proof, we need some preparations on the topology of the  state space $\cM_\rho(\mR^d)$ of SBMRE starting from $m$.

  \begin{lemma}\label{lem:compact}
  	Let $K$ be a subset of $\mathcal{M}(\mathbb{R}^d)$.
  	\begin{enumerate}[(i)]
  		\item Suppose that $\sup_{\mu\in K}\mu(f)<\infty$ holds for all $f\in C^+_c$, then $K$ is relatively compact in
  		$\mathcal{M}(\mathbb{R}^d)$;
  		
  		\item Additionally, if
  		$$
  		\lim_{R\to\infty}\sup_{\mu\in K}\<\phi_{\rho}\1_{B^c_{R}(0)},\mu\>=0,
  		$$
  		then $K$ is relatively compact in $\cM_\rho(\mR^d)$.
  	\end{enumerate}
  \end{lemma}
  \begin{proof}
  	(i).
  	Let $K_n=\{\mu|_{\bar{B}_n(0)}: \mu\in K\}$. Obviously, $K_n$ is tight in $\cM(\bar{B}_n(0))$. Therefore, $K_n$ is relatively compact in $\cM(\bar{B}_n(0))$. Then a standard diagonal argument shows that $K$ is relatively compact in $\cM(\mR^d)$ under vague topology.
  	
    \smallskip

  	(ii).
  	Let $\varphi_R\in C_c^+$  satisfying $0 \leq \varphi_R\leq 1$ and  $\varphi_R|_{B_R(0)}=1$.
  	Suppose $\{\mu_k\}_{k\ge 1}$ is a sequence in $K$ and $\{\mu_{k_j}\}_{j\geq 1}\subseteq \{\mu_k\}_{k\ge 1}$ with $k_j\to \infty$ and $\mu_{k_j}\to \mu$ in $\cM(\mR^d)$. Then we have
  	\begin{align*}
  		\<\phi_{\rho}\1_{B_R^c(0)},\mu\>\leq& \lim_{n\to\infty}\<\phi_{\rho}\1_{B_n(0)\backslash \bar{B}_{R/2}(0)}, \mu\> \leq \lim_{n\to \infty} \liminf_{j\to\infty} \<\phi_{\rho}\1_{B_n(0)\backslash \bar{B}_{R/2}(0)}, \mu_{k_j}\>\\
  		\leq& \sup_{\mu\in K} \<\phi_\rho\1_{B^c_{R/2}(0)}, \mu\> \to 0 ~ (R\to \infty),
  	\end{align*}
  	which implies $\mu\in \cM_\rho(\mR^d)$.
  	For any $f\in \Phi_\rho$, by triangle inequality
  	$$
  	  |\<f, \mu_{k_j}\>-\<f,\mu\>|\leq |\<f\varphi_R,\mu-\mu_{k_j}\>|+C \<\phi_{\rho}\1_{B^c_{R}(0)},\mu\>+ C  \sup_{\mu\in K} \<\phi_{\rho}\1_{B^c_{R}(0)},\mu \>.
  	$$
  	Letting $j\to\infty$ and then $R\to \infty$, we obtain that the left-hand side of the above inequality converges to zero, i.e.   $\mu_{k_{j}}\rightarrow \mu$ in $\cM_\rho(\mR^d)$ as $j\to\infty$. Thus, $K$ is relatively compact in $\cM_\rho(\mR^d)$.
  \end{proof}

  Let $\pi^m_t$ be the distribution of $X_t$ starting from the Lebesgue measure $m$.
  \begin{lemma}
  	Let $\rho>d$. For any $t\ge 0$, $\pi^m_t$ is supported on $\cM_\rho(\mR^d)$, and the collection of probability distributions $\{\pi^m_t, t\ge 0\}$ is tight in $\cP(\cM_{\rho}(\mathbb{R}^d))$.
  \end{lemma}
  \begin{proof}
  	We only need to prove the tightness.
  	For any $n\in\mN$ and any $\eps>0 $, let $C_n=m(B_n(0))\cdot 2^{n}\eps^{-1}$ and let $A_n=\{\mu: \mu(B_n(0))>C_n\}$.
    By the first moment formula,
  	\begin{align*}
  		\pi^m_t(A_n)&=\bP_m(X_t(B_n(0))>C_n)\leq \bE_m(X_t(B_n(0)))C_n^{-1}\\
  		&=\langle P_t \1_{B_n(0)}, m \rangle C_n^{-1}=m (B_n(0))C_n^{-1}\leq \eps 2^{-n}.
  	\end{align*}
  	Let $K_1^\eps=\bigcap_{n=1}^\infty\{\mu:\mu(B_n(0))\leq C_n\}$.  By Lemma \ref{lem:compact}, $K_1^\eps$ is a relatively compact set in
  	$\mathcal{M}(\mathbb{R}^d)$, and
  	$$\pi^m_t(K_1^\eps)= 1- \pi^m_t\l( \bigcup_{n=1}^\infty A_n \r) \geq 1-\sum_{n=1}^\infty \eps 2^{-n}\geq 1-\eps.$$
  	Since $\rho>d$, there is a constant $R(\eps, k)>0$ such that
  	$$
  	  \int_{|x|>R(\eps, k)}\phi_{\rho}(x)\d x\le\frac{\eps}{k2^k}.
  	$$
  	Put
  	$$
  	  K^\eps_2=\bigcap_{k=1}^\infty\l\{\mu: \<\phi_{\rho}\1_{B^c_{R(\eps, k)}(0)},\mu\>\leq \frac{1}{k}\r\}.
  	$$
  	Then
  	\begin{align*}
  		\pi^m_t(K^\eps_2)&\geq 1-\sum_{k=1}^\infty \pi_t^m\l(\<\phi_{\rho}\1_{B^c_{R(\eps, k)}(0)},\mu\>>\frac{1}{k}\r)\\
  		&\ge 1-\sum_{k=1}^\infty k\bE_m(X_t(\phi_{\rho}\1_{B^c_{R(\eps, k)}(0)}))=1-\sum_{k=1}^\infty k\<P_t(\phi_{\rho}\1_{B^c_{R(\eps, k)}(0)}),m\>\\
  		&=1-\sum_{k=1}^\infty k\int_{|x|>R(\eps, k)}\phi_{\rho}(x)\d x \geq 1-\eps.
  	\end{align*}
  	Set $K^\eps=K^\eps_1\bigcap K^\eps_2$. Then $\pi_t^m(K^\eps)\geq 1-2\eps$, for all $t\geq 0$.  Thanks to Lemma \ref{lem:compact}, $K^\eps$ is a relatively compact set in $\mathcal{M}_\rho(\mathbb{R}^d)$. Thus, the tightness of $\{\pi^m_t, t\ge 0\}\subseteq \cP(\cM_{\rho}(\mR^d))$ follows.
  \end{proof}
  Since $\mathcal{M}_\rho(\mathbb{R}^d)$ is a Polish space, there exists a sequence $t_n\to \infty \, (n\to\infty)$, and a probability measure $\pi^m$ on
  $\mathcal{M}_\rho(\mathbb{R}^d)$ such that $\pi^m_{t_n} \rightarrow \pi^m$ in $\cP(\cM_{\rho}(\mR^d))$ as $n\to\infty$. The next lemma shows that $\pi^m$ is
  independent of the choice of $t_n$, and thus $\pi^m_t\rightarrow \pi^m$.
  \begin{lemma}
  	Let $\rho>d$. Then $\pi_t^m$ converges to some $\pi^m$ in $\cP(\cM_\rho(\mR^d))$. Moreover,
  	\begin{equation}\label{eq:Epi-less-m}
  		\int_{\mathcal{M}(\mathbb{R}^d)}\mu\pi^m(d\mu)\leq m.
  	\end{equation}
  \end{lemma}
  \begin{proof}
  	Let $f\in C_c^+(\mR^d)$, and $u$ be the solution to \eqref{eq:spde-clt}. Then
  	\begin{equation}\label{eq:m(u)}
  		\begin{split}
  			m(u(t,\cdot))=&m(f)-\frac{1}{2}\int_{\mR^d}\d x\int_0^tP_{t-s}(u^2(s,\cdot))(x) \d s \\
  			&+\int_{\mR^d}\d x\int_0^t\int_{\mathbb{R}^d} p(t-s,x-y)u(s,y)W(
  			\d s,y)\d y=:m(f)-\xi_t.
  		\end{split}
  	\end{equation}
  	By Lemma \ref{lem:Evv}, for each $t\in [0,T]$,
  	\begin{eqnarray*}
  		&&\int_{\mathbb{R}^d} \bE \l\langle\int_0^\cdot\int_{\mathbb{R}^d} p(t-s,x-y)u(s,y)W(\d s,y)\d y \r\rangle_t \d x\\
  		&=&
  		\int_0^tds\int_{\mathbb{R}^{3d}} p(t-s,x-y)p(t-s,x-z)\bE [u(s,y)u(s,z)]\cC (y,z)\d y\d z\d x\\
  		&\leq &
  		C t \int_0^tds\int_{\mathbb{R}^{3d}} p(t-s,x-y)p(t-s,x-z) \l(P_sf (y) P_s f(z) \r)\d y\d z\d x\\
  		&= &
  		Ct \int_{\mR^{d}} \l( P_tf(x)  \r)^2  \d x<\infty.
  	\end{eqnarray*}
  	Then by the stochastic Fubini's theorem (see \cite[Theorem 4.33]{DZ2014stochastic}), we have
  	$$\xi_t=\frac{1}{2}\int_0^t ds\int_{\mathbb{R}^d}u^2(s,x)\d x-\int_0^t\!\!\int_{\mathbb{R}^d} u(s,y)W(\d s,y)\d y.$$
  	Therefore,  $\xi_t$ is a submartingale. By \eqref{eq:m(u)}, $\xi_t= m(f)-m(u(t,\cdot))\leq m(f)<\infty$. By Doob's martingale convergence theorem, $\lim_{t\rightarrow\infty}
  	\xi_t=\xi\in \mR$, $\bP $-a.s.. Therefore,
  	$$
  	\lim_{t\rightarrow\infty}\bE_m\left[\exp(-X_t(f))\right]=\lim_{t\rightarrow \infty}\bE \left[\exp(-m(u(t, \cdot))\right]=\bE \left[\exp(-m(f)+\xi)\right].
  	$$
  	So we get $\pi^m_t\rightarrow \pi^m$ in $\cP(\cM_\rho(\mR^d))$ and
  	$$
  	  \int_{\mathcal{M}(\mathbb{R}^d)}\e^{-\mu(f)}\pi^m(d\mu)=\lim_{t\rightarrow\infty}\bE_m\left[\exp(-X_t(f))\right]=\bE \left[\exp(-m(f)+\xi)\right].
  	$$
  	Using Fatou's lemma and \eqref{eq:first}, we obtain
  	$$
  	\int_{\mathcal{M}(\mathbb{R}^d)}\mu(f)\pi^m(d\mu)\leq \lim_{t\rightarrow \infty}\int_{\mathcal{M}(\mathbb{R}^d)}\mu(f)\pi^m_t(d\mu)=
  	\lim_{t\rightarrow \infty}\<P_t f,m\>=\<f,m\>,
  	$$
  	which implies \eqref{eq:Epi-less-m}.
  \end{proof}

  In order to prove Theorem \ref{thm:1}, we need to show $	\int_{\mathcal{M}(\mathbb{R}^d)}\mu\pi^m(d\mu)\geq m$. To obtain this inequality, we need to consider the time-space integration of $\bE [v^2(t, x)]$ on $\mR_+\times \mR^d$ as well as that of $\bE [v(t, x)\cC(x,y) v(t,y)]$ on $\mR_+\times \mR^d\times \mR^d$. Recall that $B_t$ and $B'_t$ are two independent $d$-dimensional Brownian motion. Let $\beta_t:=B_t-B'_t$ and $\beta'_t:=B_t+B'_t$.
  Note that $\beta_t$ and  $\beta'_t$ are two independent  $d$-dimensional Brownian motions.

  \begin{lemma}\label{lem:intEv2}
  	Let $f\in C_c^+(\mR^d)$ and $v$ be the unique solution to  \eqref{eq:PAM}. Under the same setting of Theorem \ref{thm:1},
  	\begin{equation}\label{eq:intEv2}
  		\int_0^\infty \d t \int_{\mathbb{R}^d} \bE \l[v^2(t,x)\r] \d x<\infty.
  	\end{equation}
  \end{lemma}

  \begin{proof}
  	By our assumption \eqref{cond:g}, there is a constant
	\(p\in (2d/(d-2), 2^{d-1}d/(d-2))\) so that
	\begin{align*}
		\theta'=&p\sup_{x\in\mR^{d}} \int_{\mR^{d}} G(x,y)g(y)\d y = \frac{p\Gamma\left(d/2-1\right)}{4 \pi^{d/2}}\sup_{x\in \mR^d} \int_{\mathbb{R}^d}  |x-y|^{2-d} g(y)\d y<\frac{(d-2)p}{d 2^{d-1}}<1.
	\end{align*}

  	By Khasminskii's inequality (see, e.g.,  \cite[Lemma 2.1]{Y1997dirichlet}),
  	\begin{equation}\label{eq:Eexp-pg}
  		\begin{aligned}
  			&\sup_{x}E_{x}\left[\exp \left(\int_0^\infty p g(\beta_t)\d t\right)\right]
  			\leq  \left(1- p\sup_{x} E_x \int_0^\infty g(\beta_t)\d t\right)^{-1}\\
  			=& \left(1- p\sup_{x} \int_{\mR^{d}} G(x,y)g(y)\d y\right)^{-1}<\infty.
  		\end{aligned}
  	\end{equation}
  	Now suppose $\mbox{supp}f\subset B_{N}(0)$. Let $p'=\frac{p}{p-1}$.
  	By \eqref{eq:E-v2}, H\"{o}lder's inequality and \eqref{eq:Eexp-pg},
  	\begin{equation*}
  		\begin{aligned}
  			\int_{\mR^d}\bE \l[ v^2(t,x) \r] \d x \overset{\eqref{eq:E-v2}}{=} & \int_{\mR^d} E_{(x,x)}\left[\exp \l( \int_0^t \cC (B_s, B'_s)\d s\r) (f\otimes f)(B_t, B'_t)\right] \d x \\
  			\leq& \left(E_{0} \left[\exp \l( \int_0^t p g(\beta_s)\d s\r) \right]\right)^{\frac{1}{p}} \int_{\mR^d}\left[E_{x} f^{p'}(B_t)\right]^{\frac{2}{p'}} \d x\\
  			\overset{\eqref{eq:Eexp-pg}}{\leq}& C \|f\|_\infty^2 \int_{\mR^d} \left[E_{x} \1_{B_N(0)}(B_t)\right]^{\frac{2}{p'}} \d x \\
			\leq & C \|f\|_\infty^2 \left\|p_t*\1_{B_N(0)} \right\|_{{2}/{p'}}^{{2}/{p'}}\leq C \|f\|_\infty^2 N^{\frac{2d}{p'}} \|p_t\|_{{2}/{p'}}^{{2}/{p'}} \\
			\leq& C (d, p,  \theta', N, \|f\|_\infty)  t^{-d\l(\frac{1}{p'}-\frac{1}{2}\r)}.
  		\end{aligned}
  	\end{equation*}
  	Thus,
  	\begin{equation*}
  		\begin{aligned}
  			\int_1^\infty \d t \int_{\mR^d}\bE \l[v^2(t,x)\r] \d x \leq C \int_1^\infty t^{-d(1/p+1/2)} \d t<\infty~ \l(\mbox{since } p>2d/(d-2)\r).
  		\end{aligned}
  	\end{equation*}
  	On the other hand, by Lemma \ref{lem:Evv},
  	$$
  	  \int_0^1 \d t \int_{\mathbb{R}^d} \bE \l[v^2(t,x)\r] \d x \leq C  \int_0^1\!
  	  \!\int_{\mR^d} (P_t f^2) (x)  \d x \leq C \|f\|_2^2<\infty
  	$$
  	Therefore \eqref{eq:intEv2} holds true.
  \end{proof}

  \begin{lemma}\label{lem:intEvv}
  	Let $f\in C_c^+(\mR^d)$ and $v$ be the unique solution to  \eqref{eq:PAM}. Under the same setting of Theorem \ref{thm:1},
  	\begin{equation}\label{eq:intEvv}
  		\bE \int_0^\infty \d t \int_{\mathbb{R}^{2d}}v(t,x)\cC (x,y)v(t,y)\, \d x\d y<\infty.
    \end{equation}
  \end{lemma}
 To prove Lemma \ref{lem:intEvv},
  we resort to Doob's $h$-transform. Recall that  for any $h>0$, the $h$-transformation of Brownian motion (denoted by $B^h$)
   is a diffusion process whose infinitesimal generator is given by
  $$
    L^h f := \frac{1}{2}\Delta f+\nabla \log h \cdot \nabla f.
  $$
  The Green function of $B^h$ is  $G^h(x,z) =2 h(z) G(x, z)/h(x)$, where $G$ is the Green function of $\Delta$.
  \begin{proof}[Proof of Lemma \ref{lem:intEvv}]
  	Suppose that $\mathrm{supp} f\subseteq B_{N/2}(0)$. By \eqref{eq:E-v2} and independence of $\beta$ and $\beta'$,
  	\begin{equation*}
  		\begin{aligned}
  			\bE  \left[ v(t,x)v(t,y) \right] &=  E_{(x,y)}\left[ \exp\left(\int_0^t \cC (B_s, B'_s)\d s\right) (f\otimes f)(B_t, B'_t)\right]  \\
  			&\leq E_{(x-y,x+y)}\left[\exp\left(\int_0^t g(\beta_s)\d s\right)f\otimes f\l(\frac{\beta_t+\beta'_t}{{2}}, \frac{\beta_t-\beta'_t}{{2}}\r)\right]\\
  			&\leq \|f\|_{\infty}^2  E_{(x-y,x+y)} \left[\exp\left(\int_0^t g(\beta_s)\d s\right)\1_{(B_N(0))\times B_N(0)}(\beta_t,\beta'_t )\right]\\
  			&=  \|f\|_{\infty}^2   E_{x-y}\left[\exp\left(\int_0^t g(\beta_s)\d s\right)\1_{B_{N}(0)}(\beta_t)\right]\cdot E_{x+y} \l(\1_{B_{N}(0)}(\beta'_t)\r).
  		\end{aligned}
  	\end{equation*}
  	Assume that the Feynman-Kac semigroup
  	$$
  	P^g_t \varphi(x):=  E_{x} \l[\exp\l(\int_0^t g(\beta_s)\d s\r) \varphi(\beta_t)\r]
  	$$
  	admits a Green function $G^g$. According the above estimate, we have
  	\begin{equation}\label{eq:gGg}
  		\begin{aligned}
  			&\bE \int_0^\infty \d t \int_{\mathbb{R}^{2d}}v(t,x)\cC (x,y)v(t,y)\, \d x\d y\\
  			\leq &C \int_0^\infty \d t\int_{\mathbb{R}^d} g(z) E_{z}\left[\exp\left(\int_0^t g(\beta_s)\d s\right)\1_{B_{N}(0)}(\beta_t)\right]
  			\int_{\mathbb{R}^d}E_w (\1_{B_{N}(0)}(\beta'_t))\d w\, \d z\\
  			\leq &C \int_{\mathbb{R}^d}g(z)\d z \int_0^\infty E_{z} \l[\exp\l(\int_0^t g(\beta_s)\d s\r) \1_{B_{N}(0)}(\beta_t)\r] \d t\\
  			\leq &C \int_{\mathbb{R}^d}g(x)\d x \int_{B_{N}(0)}G^{g}(x,y)\d y,
  		\end{aligned}
  	\end{equation}
  	provided that the last term of the above inequalities is finite.
  	
  	Let $P^y_x$ denote
  	the law of the $h$-Brownian motion with $h=G(\cdot,y)$. We claim that if $g$ satisfies \eqref{cond:g}, then $P^g_t$ do admit a Green function $G^g$ and
  	\begin{equation}\label{eq:Gg}
  		G^{g}(x,y)=  E_x^y \l[ \exp \l(\int_0^\infty g(\beta_s)\d s \r) \r] G(x,y)\leq C G(x,y).
    \end{equation}
    Assume \eqref{eq:Gg} holds. Plugging it into \eqref{eq:gGg}, we obtain
    \begin{equation*}
    	\bE \int_0^\infty \d t \int_{\mathbb{R}^{2d}}v(t,x)\cC (x,y)v(t,y)\, \d x\d y\leq C\int_{B_{N}(0)} \d y \int_{\mR^d} G(x,y) g(x) \d x \overset{\eqref{cond:g}}{<}\infty.
    \end{equation*}
    Therefore, it remains to prove \eqref{eq:Gg}. To do this, we first to show that
    \begin{equation}\label{eq:bridge}
    	\sup_{x,y}E_x^y \int_0^\infty g(\beta_s)\d s<1.
    \end{equation}
    In fact, by \eqref{cond:g}
    \begin{equation*}
    	\begin{aligned}
    		&\sup_{x,y}E_x^y \int_0^\infty g(\beta_s)\d s=\sup_{x,y}\int^\infty_0 E^y_x(g(\beta_s))\d s\\
    		=&\sup_{x,y}\int^\infty_0E_x(g(\beta_s)G(\beta_s, y))/G(x,y)\d s
    		=\sup_{x,y}\int_{\mathbb{R}^d}\frac{G(x,z)G(z,y)}{G(x,y)} g(z) \d z \\
    		=&  \frac{\Gamma\left(d/2-1\right)}{4 \pi^{d / 2}}  \sup_{x,y}\int_{\mathbb{R}^d}  \frac{|x-y|^{d-2}}{|x-z|^{d-2}|y-z|^{d-2}} g(z)\d z \\
    		\leq & \frac{\Gamma\left(d/2-1\right)}{4 \pi^{d / 2}}  2^{d-3} \sup_{x,y} \int_{\mR^d} (|x-z|^{2-d}+|y-z|^{2-d})g(z)\d z \\ &\qquad \qquad \qquad \qquad \qquad\quad (3G\mbox{-inequality, see \cite{CZ1995brownian}})\\
    		\leq& 2^{d-2} \sup_x\int_{\mathbb{R}^d}G(x,z)g(z)\d z\overset{\eqref{cond:g}}{<}1.
    	\end{aligned}
    \end{equation*}
  	With estimate \eqref{eq:bridge}, again using Khasminskii's inequality, we get
  	\begin{equation}\label{eq:Kh}
  		\sup_{x,y}E_x^y\l[\exp\l(\int_0^\infty g(\beta_s)\d s\r)\r]\leq \l(1-\sup_{x,y}E_x^y \int_0^\infty g(\beta_s)\d s\r)^{-1}
  		<\infty.
  	\end{equation}
  	Now we only need to prove that $E_x^y \l[ \exp \l(\int_0^\infty g(\beta_s)\d s \r) \r] G(x,y)=G^g (x,y)$.  By the Markov property of $\beta$ under $P^y_{\cdot}$,
  	\begin{equation*}
  		\begin{split}
  			E_x^y\l[\exp\l(\int_0^\infty g(\beta_s)\d s\r)\r]-1&=E_x^y\l[\int_0^\infty g(\beta_t)\exp\l(\int_t^\infty g(\beta_s)\d s\r)\d t\r]\\
  			&=E_x^y\l\{\int_0^\infty g(\beta_t)E_{\beta_t}^y\l[\exp\l(\int_0^\infty g(\beta_s)\d s\r)\r]\d t\r\}\\
  			&=\int_{\mR^d}\frac{G(x,z)G(z,y)}{G(x,y)}E_z^y\l[\exp\l(\int_0^\infty g(\beta_s)ds\r)\r] g(z) \d z.
  		\end{split}
  	\end{equation*}
  	Set $\widetilde{G^g}(,x,y)= G(x,y) 	E_x^y\l[\exp\l(\int_0^\infty g(\beta_s)\d s\r)\r]$. Then the above identity implies
  	\begin{align*}
  		\widetilde{G^g} = G + G(g \cdot \widetilde{G^g}).
  	\end{align*}
  	Taking Laplacian on both sides of the above equality, we get
  	(in the sense of distribution)
  	$$
  	  -\left(\Delta+g\right)\widetilde{G^{g}} =\mathrm{Id}.
  	$$
  	By Feynman-Kac formula, $\widetilde{G^g}$ is the Green function of the Feynman-Kac semigroup $P^g_t$, i.e.
  	\begin{align*}
  		G^g(x,y)=\widetilde{G^g}(x,y) =E_x^y\l[\exp\l(\int_0^\infty g(\beta_s)\d s\r)\r] G(x,y).
  	\end{align*}
  	So we complete our proof.
  \end{proof}

  Now we are ready to prove Theorem \ref{thm:1}.

  \begin{proof}[Proof of Theorem \ref{thm:1}]
  	For any  fixed $f\in C_c^+ (\mR^d)$,
  	let $v$ be the unique solution of \eqref{eq:PAM}, and let $u(\lambda)$ be the solution to \eqref{eq:spde-clt} with $u(\lambda;0)=\lambda f$. Then
  	\begin{equation}\label{eq:int-u-lambda}
  		\begin{aligned}
  			u(\lambda; t,x)=& \lambda P_tf(x)-\frac{1}{2}\int_0^tP_{t-s}(u^2(\lambda; s,\cdot))(x) \d s \\
  			&+\int_0^t\!\!\int_{\mathbb{R}^d} p(t-s,x-y)u(\lambda;s,y)W(\d s,y)\d y,
  		\end{aligned}
    \end{equation}
  	Integrating both sides of \eqref{eq:int-u-lambda} with respect to the Lebesgue measure $m$ on $\mR^d$ and using the stochastic Fubini's theorem (see \cite[Theorem 4.33]{DZ2014stochastic}), we get
  	\begin{equation}\label{eq:int-ul}
  	  \langle u(\lambda;t,\cdot), m\rangle=\lambda \langle f, m\rangle -\xi_t(\lambda),
  	\end{equation}
  	where
  	$$
  	  \xi_t(\lambda)=\frac{1}{2}\int_0^t\!\!\int_{\mR^d}u^2(\lambda; s,y)\d y\d s -\int_0^t\!\!\int_{\mathbb{R}^d} u(\lambda;s,y)W(\d s,y)\d y.
  	$$
  	By condition \eqref{cond:g}, Lemma \ref{lem:intEv2} and the fact that $u(\lambda)\leq \lambda v$, one sees that
  	  	$$
  	  \eta(\lambda):=\frac{1}{2}\int_0^ \infty\!\!\int_{\mathbb{R}^d}u^2(\lambda; t,x)\d x\d t<\infty.
  	$$
  	Similarly, Lemma \ref{lem:intEvv} and  It\^o isometry \eqref{eq:Ito} yield
  	$$
  	  m(\lambda):=-\int_0^\infty \!\! \int_{\mathbb{R}^d} u(\lambda; s,x)W(\d s,x)\d x
  	$$
  	is well defined. Put $\xi(\lambda):=\eta(\lambda)+m(\lambda)$.  Letting $t\to \infty$ in \eqref{eq:int-ul}, we obtain
  	\begin{equation}\label{eq:upper-zeta}
  		0\leq \lim_{t\to\infty}\langle u(\lambda;t,\cdot), m\rangle=\lambda \langle f,m\rangle -\xi(\lambda),
  	\end{equation}
  	which implies $\xi(\lambda)\leq \lambda \langle f, m\rangle $. By the dominated convergence theorem, we have
  	\begin{equation}\label{eq:Eexp-lambdaf}
  		\int_{\mathcal{M}(\mathbb{R}^d)}\e^{-\mu(\lambda f)}\pi^m(\d\mu)=\bE \left[\exp(-\lambda \langle f, m \rangle +\xi(\lambda))\right].
  \end{equation}
  	Taking derivation with respect to $\lambda$ on both sides of \eqref{eq:Eexp-lambdaf} at $\lambda=0$, again by the dominated convergence theorem, we have
  	\begin{equation}\label{eq:Epim}
  		\int_{\mathcal{M}(\mathbb{R}^d)}\mu(f)\pi^m(\d \mu)=\langle f, m \rangle -\lim_{\lambda\rightarrow 0}\lambda^{-1}(\bE \e^{\xi(\lambda)}-1).
    \end{equation}
    Thus, it remains to show
    $$
      \lim_{\lambda\rightarrow 0}\lambda^{-1}(\bE \e^{\xi(\lambda)}-1)=0.
    $$
  	Since $\p_\lambda u(\lambda; t, x)\leq v(t,x)$ (see \eqref{eq:u'-less-v}), by Lemmas \ref{lem:intEv2}, \ref{lem:intEvv} and the dominated convergence theorem,  we have
  	\begin{equation}\label{eq:eta'}
  		\p_\lambda \eta(\lambda)=\int_0^\infty\!\!\int_{\mathbb{R}^d} u(\lambda;t,x)\p_\lambda u(\lambda;t,x)\d x\d t\leq\lambda\int_0^\infty\!\!\int_{\mathbb{R}^d}v^2(t,x)\d x\d t
  	\end{equation}
    and
  	\begin{equation*}
  		\p_\lambda m(\lambda)=-\int_0^\infty\!\!\int_{\mathbb{R}^d} \p_\lambda u(\lambda;t,x)W(\d t,x)\d x,
    \end{equation*}
    and each of the right-hand side of the above two inequalities is well-defined real variables. Note that
  	\begin{equation}\label{eq:Dzeta}
  		\lambda^{-1}(\bE \e^{\xi(\lambda)}-1)=\lambda^{-1} \int_0^\lambda \l[\bE (\e^{\xi(\tau)}\p_\lambda \eta(\tau))+\bE (\e^{\xi(\tau)}\p_\lambda m(\tau))\r]\d \tau=:I_1(\lambda)+I_2(\lambda).
   \end{equation}
   Below we are going to prove $I_i(\lambda) \to 0$ as $\lambda\to 0$ $(i=1,2)$. For $I_1$, by \eqref{eq:upper-zeta} and \eqref{eq:eta'},
  	one can see that for any $\lambda\in (0, 1)$,
  	$$
  	  0\leq I_1(\lambda) \leq
  	  \lambda \e^{\lambda \langle f, m \rangle }\int_0^\infty\!\!\int_{\mathbb{R}^d}\bE v^2(t,x)\d t\d x.
  	$$
  	Hence,
  	\begin{equation}\label{eq:Deta}
  		\lim_{\lambda\rightarrow0^+}\lambda^{-1} \int_0^\lambda \bE \e^{\xi(\tau)}\p_\lambda \eta(\tau) \d \tau =0.
  	\end{equation}
   For $I_2(\lambda)$, set
  	$$
  	  V :=\lim_{t\rightarrow\infty}\langle v(t), m\rangle =\langle f, m \rangle +\int_0^\infty\!\! \int_{\mathbb{R}^d}v(s,z)W(\d s,z)\d z,
  	$$
  	which is finite a.s., due to Lemma \ref{lem:intEvv}. By \eqref{eq:int-ul}, $\xi_t(\lambda)\geq \lambda \langle f, m \rangle -\lambda \langle v(t),m\rangle$, so
  	\begin{equation}\label{eq:lower-zeta}
  		\xi(\lambda)\geq \lambda \langle f, m \rangle -\lambda V .
    \end{equation}
  	Since $\bE  \p_\lambda m(\tau)=0$, for any $M>0$, we have
  	\begin{equation}\label{eq:def-J}
  		\begin{aligned}
  			I_2(\lambda)=&\lambda^{-1} \int_0^\lambda \bE (\e^{\xi(\tau)}\p_\lambda m(\tau))\d\tau-\lambda^{-1} \int_0^\lambda \bE \p_\lambda m(\tau)\d\tau\\
  			\leq &\lambda^{-1}\int_0^\lambda \bE \left(|\e^{\xi(\tau)}-1||\p_\lambda m(\tau)|; ~V \leq M\right)\d\tau\\
  			&+\lambda^{-1}\int_0^\lambda \bE \left(|\e^{\xi(\tau)}-1||\p_\lambda m(\tau)|; ~  V > M\right)\d\tau\\
  			=&:I_{21}(\lambda)+I_{22}(\lambda).
  		\end{aligned}
  	\end{equation}
  	For any $\varepsilon>0 $, choose $M>3\langle f, m \rangle $ such that
  	$$
  	  \bP (V >M)\leq \varepsilon.
  	$$
  	Thanks to \eqref{eq:upper-zeta}, for  $\lambda$ sufficiently small, we have  $\e^{\xi(\tau)}-1\leq 2\xi^+(\tau)\leq 2\lambda \langle f, m \rangle $ for  all $\tau\in [0, \lambda]$. On the set $\{V \leq M\}$, using \eqref{eq:lower-zeta}, we have
  	$$
  	  1-\e^{\xi(\tau)}\leq 1-\e^{\tau(\langle f, m \rangle -V )}\le 1-\e^{\tau(\langle f, m \rangle -M)}\le\lambda (M-\langle f, m \rangle )\leq \lambda M,\quad\forall \tau\in[0, \lambda].
  	$$
  	 So on the set $\{V \leq M\}$, $|\e^{\xi(\tau)}-1|\leq \lambda M$ for  all $\tau\in [0, \lambda]$. This implies that
  	\begin{equation}\label{eq:upper-J1}
  		I_{21}(\lambda)\le \lambda M\sup_{\tau>0}\bE |\p_\lambda m(\tau)| \leq \lambda M \sup_{\tau>0} (\bE |\p_\lambda m(\tau)|^2)^{1/2}.
  	\end{equation}
  	Using H\"{o}lder's inequality and \eqref{eq:upper-zeta}, we also have
  	\begin{equation}\label{eq:upper-J2}
  		I_{22}(\lambda)\le  \sqrt{\varepsilon} (1+\e^{\lambda \langle f, m \rangle })\sup_{\tau>0} (\bE |\p_\lambda m(\tau)|^2)^{1/2}.
  	\end{equation}
  	Noting that
  	\begin{align*}
  		\bE  \p_\lambda m(\tau)^2 =&\bE \l[\int_0^\infty\!\! \int_{\mathbb{R}^d}\p_\lambda u(\tau,t,x)W(\d s,x)\d x\r]^2\\
  		\overset{\eqref{eq:Ito}}{=}&\bE \int_0^\infty \d t \int_{\mathbb{R}^{2d}}  \p_\lambda u(\tau,t,x) \p_\lambda u(\tau,t,y)\cC (x,y) \d x \d y\\
  		\leq& \bE \int_0^\infty \d t \int_{\mathbb{R}^{2d}} v(t,x)v(t,y)\cC (x,y) \d x \d y\overset{\eqref{eq:intEvv}}{\leq }C ~ (C \mbox{ does not depend on }\tau)
  	\end{align*}
  	and combining \eqref{eq:def-J}-\eqref{eq:upper-J2}, we get
  	$$
  	  I_2(\lambda) \le  C \l[ \lambda M +\sqrt{\varepsilon} (1+\e^{\lambda \langle f, m \rangle }) \r].
  	$$
  	 Letting $\lambda\rightarrow 0$ and then $\varepsilon\rightarrow 0$, we obtain $\lim_{\lambda\rightarrow0^+}I_{2}(\lambda)=0$. This together with  \eqref{eq:Epim}, \eqref{eq:Dzeta} and \eqref{eq:Deta} yields
  	$$
  	  \int_{\mathcal{M}(\mathbb{R}^d)}\mu(f)\pi^m(d\mu)=\langle f, m \rangle -\lim_{\lambda\rightarrow 0}\lambda^{-1}(\bE \e^{\xi(\lambda)}-1)=\langle f, m \rangle ,\quad\forall f\in C_c^+.
  	$$
  	So we complete the proof for the first conclusion.
  	
  	Next we show that $\pi^m$ is an invariant measure of the $\cM_\rho(\mR^d)$-valued Markov process $X$.
  	The proof for this is in fact standard.  For any $f\in C_c^+$, by the Markov property of $X$,
  	\begin{equation}\label{eq:Markov}
  		\bE_m \e^{-\langle f,X_{t+s}\rangle} = \bE_m \bE_{X_t} \e^{-\langle f,X_s\rangle}=	\int_{\mathcal{M}(\mathbb{R}^d)} \bE_{\mu} \e^{-\langle f,X_s\rangle} \pi^m_t(d\mu).
  	\end{equation}
  	We claim that the function $\mu\mapsto \bE_{\mu} \e^{-\langle f,X_s\rangle}$ is a continuous function from $\cM_\rho(\mR^d)$ to \(\mR\). Assuming this, then using Theorem \ref{thm:1} and letting $t\to\infty$ in \eqref{eq:Markov}, we obtain that
  	\begin{align*}
  		\int_{\mathcal{M}(\mathbb{R}^d)}\e^{-\<f,\mu\>}  \pi^m(d\mu)=&\lim_{t\to\infty} \bE_m \e^{-\langle f,X_{t+s}\rangle}=	\lim_{t\to\infty}\int_{\mathcal{M}(\mathbb{R}^d)} \bE_{\mu} \e^{-\langle f,X_s\rangle} \pi^m_t(d\mu) \\
  		=& \int_{\mathcal{M}(\mathbb{R}^d)} \bE_{\mu} \e^{-\langle f,X_s\rangle} \pi^m(d\mu),
  	\end{align*}
  	i.e $\bE_{\pi^m}\e^{\<f,X_0\>}=\bE_{\pi^m}\e^{\<-f,X_s\>}$, which means $\pi^m$ is an invariant distribution of $X$. Now we prove that the function $\mu\mapsto \bE_{\mu} \e^{-\langle f,X_s\rangle}$ is a continuous function in $\cM_\rho(\mR^d)$. Suppose that $\mu_n\rightarrow \mu$ in $\cM_\rho(\mR^d)$, then
  	$$\l|\bE_{\mu_n} \e^{-\langle f,X_s\rangle}-\bE_{\mu} \e^{-\langle f,X_s\rangle}\r|=\l|\bE \e^{-\<u(s),\mu_n\>}-\bE \e^{-\<u(s),\mu\>}\r|\leq \bE \l|1-\e^{-\<u(s),\mu_n-\mu\>}\r|.$$
  	Thanks to \eqref{eq:v-decay}, there is a positive random variable $K$ such that $0\leq u(s)\leq v(s) \leq K\phi_{\rho}$, for each $\rho>0$. Therefore $\<u(s),\mu_n-\mu\>\rightarrow 0$, $\bP$-a.s. By the dominated convergence theorem, one sees that $\bE \l|1-\e^{-\<u(s),\mu_n-\mu\>}\r|\rightarrow 0$, which yields the continuity of the map $\mu\mapsto \bE_{\mu} \e^{-\langle f,X_s\rangle}$.
  \end{proof}

  \section{Local extinction in ``strong" environments}\label{Sec:extincition}

  In this section, we give the proof for Theorem \ref{thm:2}.

  We use $\zeta$ to denote the Gaussian noise with spatially homogeneous correlation function $\Theta$:
    $$
    \bE [ \zeta (s,x)\zeta(t,y) ] = \Theta(x-y) (t\wedge s),
  $$
  where $\Theta\in C^\beta(\mR^d)$  with $\beta>1$  and  $\Theta(0)=1$.
  In the following, we always assume that $W=\sqrt{a} \zeta$, and under $\bP_m$,
  $X$ is the solution of \eqref{eq:MP}  with $X_0=\mu$. Then
  $$
  \bE_\mu \exp(-X_t(f))=\bE  \exp(-\langle u(t),\mu\rangle),
  $$
  and $u$ is the solution to \eqref{eq:spde-clt}.  Let $v$ and $\widetilde{v}$ be solutions to
  $$
    \p_t v = \frac{1}{2} \Delta v+ v \, \p_tW ~\mbox{ and }~ \p_t \widetilde{v} = \frac{1}{2} \Delta \widetilde{v}+ \widetilde{v} \circ \p_t W
  $$
  with $v(0)=\widetilde{v}(0)=1$, respectively.
  Here $\circ W(\d s,x)$ denotes the Stratonovich differential of $W(s,x)$. By our assumption \eqref{cond:Gamma},
  \begin{equation}\label{id-v}
  	\widetilde{v} (t,x)=v (t,x)\e^{\frac{a t}{2}}.
  \end{equation}
 Let \(E\) and \(P\) denote the expectation and probability with respect to the Brownian motion \(B\), which is independent of \(\zeta\). The following stochastic Feynman-Kac formula (see  \cite{CV1998almost}) will be used frequently: for fixed $t\geq 0$ and $x\in\mR^d$,
  \begin{equation}\label{eq:SFK}
  	\widetilde{v}(t,x)=E\left[f(x+B_t)\exp \l(\int_0^tW(\d s,x+B_{t-s})\r)\right],\quad \forall f\in C_b(\mR^d). \tag{SFK}
  \end{equation}

  Our proof for Theorem \ref{thm:2} is based on the following large deviation result whose proof is  presented in Appendix \ref{Sec:App-LDP}.
  \begin{lemma}\label{lem:LDP}
  	There are constants $N_0(d,\Theta)\geq 1$ such that for any $a\geq N_0$ and sufficient large $t$,
  	$$
  	\bP \left(\sup_{|x|\leq t} v(t,x) > \e^{-\frac{a t}{3}} \right)\leq  \e^{-c t},
  	$$
  	where $c>0$ only depends on $d,a$ and $\Theta$.
  \end{lemma}

  \medskip

  \begin{proof}[Proof of Theorem \ref{thm:2}]
  	Take $a\geq N_0\geq 1$.
  	Define
  	$$
	m_n(A)=m(B_n(0)\cap A) ~\mbox{ and }~ m'_n(A)=m(B^c_n(0)\cap A).
	$$

  	Let $u^k$ be the solution of equation \eqref{eq:spde-clt} with $u_0=k$, and  \(w^k\) be the solution to
  	$$
  	 \p_tw^k(t)= \frac{1}{2}\Delta w^k(t)-\frac{1}{2}(w^k(t))^2,\quad w^k(0)=k.
  	$$
	Noting that the map \(f\mapsto \e^{-\<f, \mu\>}\) is concave, by Jensen's inequality and \eqref{eq:clt}, we have
	   \begin{equation}\label{eq:by-Jensen}
  	\bE_{\mu}	\e^{-\<k,X_{n/2}\>}= \bE_{\mu} \left( \bE_{\mu}  ( \e^{-\<k,X_{n/2}\>} | W ) \right) = \bE \e^{-\< u^k_{n/2},\mu\>}\geq \e^{-\< \bE u^k_{n/2},\mu\>}.
        \end{equation}
  Taking expectation on both side of equation \eqref{eq:spde-clt},
   and defining $\alpha := \bE [u^k]$, we have by the Cauchy-Schwarz inequality that
  	$$
  	  \partial_t \alpha = \frac{1}{2}\Delta \alpha - \frac{1}{2}\alpha^2 + \frac{1}{2} (\bE u^k)^2 - \frac{1}{2}\bE [(u^k)^2] \leq \frac{1}{2}\Delta \alpha - \frac{1}{2}\alpha^2,
  	$$
 in the sense of distribution. Thus, by the comparison principle (see Lemma \ref{lem:comparion} and Remark \ref{rmk:bdd-initial}), we conclude that $\bE [u^k(t)] \leq w^k(t)$. Note that $w^k(t)$ is spatially homogeneous, i.e. $w^k(t,x)=w^k(t)$, one sees that
  	$$
  	  w^k(t)=\frac{1}{t/2+1/k}.
  	$$	
  	Hence
  	$$
  	  \bE [ u^k(t) ] \leq w^k(t)=\frac{1}{t/2+1/k}.
  	$$
  	This to together with \eqref{eq:by-Jensen} and the Markov property of $X$ yield that
  	\begin{equation}\label{eq:Prob_Xn}
	\begin{aligned}
		\bP_{m_n} (X_{n}=0)=&\lim_{k\rightarrow \infty} \bE_{m_n}\e^{-\<k, X_{n}\>}=\lim_{k\rightarrow \infty} \bE_{m_n}\bE_{X_{n/2}}\e^{-\<k,X_{n/2}\>}\\
		\geq& \lim_{k\rightarrow \infty} \bE_{m_n}\e^{-\frac{1}{n/4+1/k}\<1,X_{n/2}\>}\geq \lim_{k\rightarrow \infty} \bE \e^{-\frac{1}{n/4+1/k}\l\<v(\frac{n}{2}),m_n\r\>}\\
		=& \bE \exp\left(-\frac{4}{n}\int_{|x|\leq n}v(n/2,x)\d x\right).
	\end{aligned}
	\end{equation}
	Now set
  	$$
  	  A(n)=\l\{\int_{|x|\leq n}v(n,x)\d x\leq \e^{-\frac{a}{4}n} \r\}.
  	$$
	By Lemma \ref{lem:LDP}, for large $n$,
  	\begin{equation}\label{eq:P(A(n)}
  		\bP (A(n))\ge 1-\e^{-c n},
  	\end{equation}
  	where $c=c(d, a, \Theta)$. Therefore, by \eqref{eq:Prob_Xn} and \eqref{eq:P(A(n)}, we have
	\begin{align*}
  		\bP_{m_n}(X_{n}\neq0)=&1-\bP_{m_n}(X_n= 0)\\
		\leq& 1- \bE \left[\exp\left(-\frac{4}{n}\int_{|x|\leq n}v(n/2,x)\d x\right); A(n)\right]\\
		\leq& 1-\bP (A(n))\exp\left(-4\e^{-\frac{a}{8}n}\right)\leq \e^{-c n}.
  	\end{align*}

  	According to \cite[Lemma 4.1]{MX2007local}, there is  probability space $\l(\Omega,\mathcal{F},\bP\r)$ and common noise $W$ such that $(X^n, W)$ and $(\overline{X}^n, W)$ are solutions to \eqref{eq:CMP}  with initial measure $m_n$ and $m'_n$, respectively. Moreover, $X^n$ and $\overline{X}^n$ are independent  under conditional probability $\bP^W$. By the proof for \cite[Theorem 1.1]{MX2007local}, there exists $\delta'>0$ such that
  	$$
  	  \bP\l(\int_0^{n+1}\overline{X}^{n}_s(K)ds\neq 0\r)\leq C \e^{-n^{\delta'}}.
  	$$
  	Then we have
  	\begin{align*}
  		&\bP\l(X_t(K)\neq 0, ~ \exists t\in [n,n+1]\r)\leq \bP \l(X^{n}_{n}\neq0 \,\,\,\mbox{ or } \int_0^{n+1}\overline{X}^{n}_s(K)ds\neq0\r)\\
  		\leq& \bP\l(X^n_{n}\neq0\r)+\bP\l(\int_0^{n+1}\overline{X}^{n}_s(K)ds\neq 0\r)\leq \e^{-c n}+C \e^{-n^{\delta'}},
  	\end{align*}
  	Thanks to Borel-Cantelli's Lemma, for any compact set $K$
  	$$
  	  \bP\l( \l\{X_t(K)\neq 0, ~ \exists t\in [n,n+1] \r\} ~ i.o. \r)=0.
  	$$
  	This implies that there exists a random time $T_{K}$, such that
  	$$
  	  X_t(K)=0 \quad \hbox{for }  t\geq T_K.
  	$$
  \end{proof}

\appendix

\section{Uniqueness of Martingale Problem with infinite initial measure}\label{Sec:App-MP}
\setcounter{equation}{0}
\renewcommand\theequation{A.\arabic{equation}}

In this section, we follow the idea of \cite[Section 4]{Myt1996environments} to present a complete proof for the uniqueness of the martingale problem \eqref{eq:MP} with infinite initial measure \(\mu\in \cM_{\rho}(\mR^d)\). The proofs are similar to those in \cite[Section 4]{Myt1996environments}, but  there are several key points  that need to be addressed and details supplied.

\medskip

The following uniform  moment boundedness result is an analog
of \cite[Lemma 4.11]{Myt1996environments}.
\begin{lemma}\label{lem:moment-bound}
Let $\rho>0$, $k\geq 1$, and let $X$ be any solution to \eqref{eq:MP} with initial measure $\mu \in \cM_\rho(\mathbb{R}^d)$. For any $T>0$, it holds that
$$
\sup_{t\in[0,T]} \bE\left[\<\phi_\rho, X_t\>^k\right] <\infty .
$$
\end{lemma}

\begin{proof}
	The martingale problem is formulated using test functions in $C_c^2(\mathbb{R}^d)$. Since $\phi_\rho$ does not have compact support, we introduce a standard cutoff sequence. Let $\chi_R \in C_c^\infty(\mathbb{R}^d)$ such that $0 \leq \chi_R \leq 1$, $\chi_R(x) = 1$ for $|x| \leq R$, $\chi_R(x) = 0$ for $|x| \geq R+1$, with globally bounded derivatives $|\nabla \chi_R| \leq 2$ and $|\Delta \chi_R| \leq 4$. Define the compactly supported test function $f_R(x) := \phi_\rho(x)\chi_R(x) \in C_c^\infty(\mathbb{R}^d)$. Using the product rule, we have $\Delta f_R(x) \leq C \phi_\rho(x)$,
where $C$ is a positive constant.

Define the stopping time $\tau_N = \inf\{ t\geq 0 : \<\phi_\rho, X_t\> \geq N \}$. Applying It\^o's formula to the continuous semimartingale $\<f_R, X_{t\wedge\tau_N}\>^k$, we get from \eqref{eq:MP} that
\begin{align*}
\<f_R, X_{t\wedge\tau_N}\>^k =& \<f_R, \mu\>^k + \frac{k}{2}\int_0^{t\wedge\tau_N} \<f_R, X_s\>^{k-1} \<\Delta f_R, X_s\> \mathrm{d}s \\
&+ \frac{1}{2}k(k-1)\int_0^{t\wedge\tau_N} \<f_R, X_s\>^{k-2} \mathrm{d}\langle M_t^{f_R} \rangle_s + \text{Martingale} ,
\end{align*}
where $M_t^{f_R} $ is the martingale defined by \eqref {eq:MP} with $f_R$ in place of $f$ there.
Taking expectations, we obtain
\begin{align*}
\bE\left[\<f_R, X_{t\wedge\tau_N}\>^k\right] & \leq \<\phi_\rho, \mu\>^k + C k\int_0^t \bE\left[\<\phi_\rho, X_{s\wedge\tau_N}\>^k\right] \mathrm{d}s \\
& + \frac{1}{2}k(k-1) \int_0^t \bE\left[ \<\phi_\rho, X_{s\wedge\tau_N}\>^{k-2} \left(
\<\phi_\rho^2, X_{s\wedge\tau_N}\>
+ \|\cC\|_\infty \<\phi_\rho, X_{s\wedge\tau_N}\>^2 \right) \right] \mathrm{d}s.
\end{align*}
Noticing that $\phi_\rho^{2}\leq\phi_\rho$, letting $R\to\infty$, we get
$$ \bE\left[X_{t\wedge\tau_N}(\phi_\rho)^k\right] \leq \<\phi_\rho, \mu\>^k +
C(k)T+C(k)
\int_0^t \bE\left[\<\phi_\rho, X_{s\wedge\tau_N}\>^k\right] \mathrm{d}s. $$
By Gronwall's inequality, $\bE[X_{t\wedge\tau_N}(\phi_\rho)^k]$ is bounded uniformly in $N$. Letting $N \to \infty$ and applying Fatou's lemma yields the desired uniform bound.
\end{proof}

 Let $\xi$ be a random fields on $\mR^d$ with correlation function 
$\cC (x, y)$. Define $\xi^{(n)}(x)$ as the truncation of $\xi$ (see \cite[p.1968]{Myt1996environments}):
$$
\xi^{(n)}(x):=\left\{\begin{array}{ll} \sqrt{n},\quad &
\xi(x)>\sqrt{n},\\
-\sqrt{n}, & 
\xi(x)<-\sqrt{n},\\
\xi(x), & \mbox{otherwise}.\end{array}\right.
$$
  Then $\cC_n(x,y):=\bE \left[ \xi^{(n)}(x)\xi^{(n)}(y) \right]$ and  $\bE[\xi^{n)}(x)\xi^{(n)}(y)\xi^{(n)}(z)]$, as function on
  $\mR^{2d}$  and $\mR^{3d}$ respectively,  are uniformly bounded in $n$.
Let
\(0\leq \phi \in \Phi_\rho\) and \(Y^{(n)}_t\) be the same rcll process on $[0, \infty)$ as defined in   \cite[p.1969]{Myt1996environments},
taking values in the space of non-negative Borel measurable bounded functions on $\R^d$ equipped with the uniform norm.
  By the discussion therein, \(Y^{(n)}\) is the unique solution to the following stochastic evolution equation:
\[
Y^{(n)}_t(x)=P_t\phi (x)-\int_0^t P_{t-s} \left(\frac{1}{2}{Y_s^{(n)}}^2 \right)(x) \d s +\int_0^t \int_{\cB_b(\mR^d)} P_{t-s}\left( \frac{h}{\sqrt{n}}Y^{(n)}_{s-}\right)(x) N^{(n)}(\d s, \d h),
\]
where \(N^{(n)}\) is a Poisson random measure on \(\mR_+\times \cB_b(\mR^d)\) with intensity measure \(n\d s \times \bP_{\xi^{(n)}}(\d h)\),
where $\bP_{\xi^{(n)}}(\d h)$ denotes the distribution of $\xi^{(n)}$.

\medskip

The following moment estimates for the dual process is an improvement of Lemma 4.15 in \cite{Myt1996environments}, which is crucial for the uniqueness proof of the martingale problem with infinite initial measure.
\begin{lemma}[Third Moment Estimate for the Dual Process]\label{lem:dual_3moment}
For any $T>0$, there exists $C(T) > 0$
independent of $n$ and $x$ such that
$$
\sup_{n\geq 1}\sup_{t\in[0,T]} \bE \left[(Y_t^{(n)}(x))^3\right] \leq C(T) \phi^3_\rho(x) \quad \hbox{ for every }  x\in \mR^d.
$$
Consequently,
$$
\sup_{n\geq 1}\sup_{t\in[0,T]} \bE \left[ Y_t^{(n)}(x) Y_t^{(n)}(y) \right]\leq C(T) \phi_\rho(x) \phi_\rho(y)
\quad \hbox{ for every }  x,y\in \mR^d.
$$
\end{lemma}

\begin{proof}
Define
\begin{align*}
V_r(x) &:= P_{t-r} Y_r^{(n)}(x) \\
&= P_t\phi(x) - \int_0^r P_{t-s}\left(\frac{1}{2}(Y_s^{(n)})^2\right)(x) \mathrm{d}s + \int_0^r \int_{\cB_b(\mR^d)} P_{t-s}\left(\frac{h}{\sqrt{n}}Y_{s-}^{(n)}\right)(x) N^{(n)}(\mathrm{d}s, \mathrm{d}h).
\end{align*}
For any fixed $x \in \mR^d$, $V_r(x)$ is a standard semimartingale on $r \in [0, t]$.

We apply It\^o's formula to the function $F(x) = x^3$
for
the process $V_r(x)$,  evaluate it up to $r=t$, and take the expectation. Notice that $V_t(x) = Y_t^{(n)}(x)$ and $V_0(x) = P_t\phi(x)$. Replacing the Poisson random measure with its compensator $\widehat{N}^{(n)}(\mathrm{d}s, \mathrm{d}h) = n \mathrm{d}s \otimes \mathbf{P}_{\xi^{(n)}}(\mathrm{d}h)$, we obtain:
\begin{align*}
\bE \left[(Y_t^{(n)}(x))^3\right] &= (P_t\phi(x))^3 - 3 \int_0^t \bE \left[ (V_s(x))^2 \cdot P_{t-s}\left(\frac{1}{2}(Y_s^{(n)})^2\right)(x) \right] \mathrm{d}s \\
&\quad + n \int_0^t \bE \left[ \bE _{\xi^{(n)}}\left[ (V_s(x) + B_s(x))^3 - (V_s(x))^3 \right] \right] \mathrm{d}s,
\end{align*}
where $B_s(x) := \frac{1}{\sqrt{n}}P_{t-s}(\xi^{(n)} Y_s^{(n)})(x)$,
$\bE _{\xi^{(n)}}$ denotes expectation with respect to $\bP_{\xi^{(n)}}(\d h)$.
 Since $Y_s^{(n)} \geq 0$ and $P_{t-s}$ preserves positivity,
 the second term of the right hand side of the above equality is non-positive, so we drop it for an upper bound.

Note that
$(V_s + B_s)^3 - V_s^3 = 3V_s^2 B_s + 3V_s B_s^2 + B_s^3$.

\begin{enumerate}[(i)]
    \item \textit{The $3V_s^2 B_s$ term:}
    Since $\bE _{\xi^{(n)}}[\xi^{(n)}] = 0$, we may change the order of $\bE _{\xi^{(n)}}$ and  $P_{t-s}$, and then  we get $0$ for this term.

    \item \textit{The $3V_s B_s^2$ term:}
    We apply Jensen's inequality $(P_{t-s} f)^2 \leq P_{t-s}(f^2)$:
    \begin{align*}
    \bE _{\xi^{(n)}}\left[ B_s(x)^2 \right] &= \frac{1}{n} \bE _{\xi^{(n)}}\left[ (P_{t-s}(\xi^{(n)} Y_s^{(n)})(x))^2 \right] \\
    &\leq \frac{1}{n} \bE _{\xi^{(n)}}\left[ P_{t-s}( (\xi^{(n)})^2 (Y_s^{(n)})^2 )(x) \right] \\
    &\leq \frac{C}{n} P_{t-s}\left( (Y_s^{(n)})^2 \right)(x),
    \end{align*}
where in the last inequality we used the uniform boundedness of  $\cC_n(x,y)$.
    \item \textit{The $B_s^3$ term:} Similarly, using $(P_{t-s} f)^3 \leq P_{t-s}(f^3)$  for $f\geq 0$:
    $$ \bE _{\xi^{(n)}}\left[ B_s(x)^3 \right] \leq \frac{1}{n^{3/2}} P_{t-s}\left( \bE _{\xi^{(n)}}[|\xi^{(n)}|^3] (Y_s^{(n)})^3 \right)(x) \leq \frac{C}{n^{3/2}} P_{t-s}\left( (Y_s^{(n)})^3 \right)(x). $$
\end{enumerate}

Substituting these bounds back into the expectation integral, the factor $n$ cancels out:
\begin{align}\label{int-Y3}
\bE \left[(Y_t^{(n)}(x))^3\right] \leq & (P_t\phi(x))^3 + C \int_0^t \bE \left[ V_s(x) P_{t-s}\left( (Y_s^{(n)})^2 \right)(x) + n^{-\frac{1}{2}} P_{t-s}\left( (Y_s^{(n)})^3 \right)(x) \right] \mathrm{d}s.
\end{align}
By H\"older's inequality,
$$ V_s(x) = P_{t-s}(Y_s^{(n)})(x) \leq \left[ P_{t-s}\left( (Y_s^{(n)})^3 \right)(x) \right]^{1/3}, $$
$$ P_{t-s}\left( (Y_s^{(n)})^2 \right)(x) \leq \left[ P_{t-s}\left( (Y_s^{(n)})^3 \right)(x) \right]^{2/3}.
$$
 Multiplying these together, we get the elegant deterministic bound:
$$ V_s(x) \cdot P_{t-s}\left( (Y_s^{(n)})^2 \right)(x) \leq P_{t-s}\left( (Y_s^{(n)})^3 \right)(x). $$

Let $u(t, x) := \bE [(Y_t^{(n)}(x))^3]$. Fubini's theorem implies $P_{t-s}(\bE [f]) = \bE [P_{t-s}(f)]$. Thus,
the inequality \eqref{int-Y3} simplifies to
\begin{equation}\label{int-u}
u(t, x) \leq (P_t\phi(x))^3 + C \int_0^t P_{t-s}( u(s, \cdot) )(x) \mathrm{d}s.
\end{equation}
Let us define the weighted supremum norm of $u(t, \cdot)$ as:
$$
v(t) := \sup_{x \in \mR^d} u(t, x)\phi_\rho^{-3}(x), \quad t \in [0, T].
$$
By \eqref{int-u} and \eqref{eq:Ptphi}, we have
\begin{align*}
u(t, x) &\leq C \phi_\rho^3(x) + C \int_0^t P_{t-s}\left( v(s) \phi_\rho^3 \right)(x) \mathrm{d}s \\
&\leq
C(T)\phi_\rho^3(x) + C(T)
 \int_0^t v(s) \phi_\rho^3(x) \mathrm{d}s,
\end{align*}
i.e. \(v(t) \leq C(T) + C(T) \int_0^t v(s) \d s\).
Applying Gronwall's inequality, we conclude $u(t, x) \leq C(T)\phi_\rho^3(x)$, which completes the proof.
\end{proof}

The previous lemma ensures that $Y^{(n)}_t \in \Phi_\rho$ for each $t \geq 0$ whenever $\phi \in \Phi_\rho$, a property essential for the well-definedness and integrability of the duality identity in Lemma \ref{lem:dual_identity} below. This identity is obtained by applying
an  analysis result \cite[Lemma 4.17]{Myt1996environments} to the function
 $f(s,t) = \mathbf{E}\left[\exp\left(-X_s(P_{T-t}Y_t^{(n)})\right)\right]$. As the argument is
analogous to that of \cite[Lemma 4.18]{Myt1996environments} by using the estimate in Lemma \ref{lem:dual_3moment},
we omit the proof and proceed directly to the statement of the identity.

\begin{lemma}[Duality Identity]\label{lem:dual_identity}
Let $X$ solve \eqref{eq:MP} with $\mu\in \cM_\rho(\mR^d)$, independent of $Y^{(n)}$. For any $t\geq 0$ and $\phi\in \Phi_\rho$,
\begin{equation}\label{eq:dual_identity}
	\begin{aligned}
		& \bE\left[\exp(-\<\phi, X_t\>)\right] - \bE\left[\exp(-X_0(Y_t^{(n)}))\right] \\
		=& \bE\left[ \int_0^t \exp\left(-X_s(Y_{t-s}^{(n)})\right) \cdot \left( \frac{1}{2}\int_{E\times E} \cC(x,y)Y_{t-s}^{(n)}(x)Y_{t-s}^{(n)}(y) X_s(\mathrm{d}x)X_s(\mathrm{d}y) \right. \right. \\
		& \left. \left. - n\bE_{\xi^{(n)}}\left[ \exp\left(-X_s\left(\frac{\xi^{(n)}}{\sqrt{n}}Y_{t-s}^{(n)}\right)\right) - 1  \right] \right) \mathrm{d}s \right].
	\end{aligned}
\end{equation}
\end{lemma}

\begin{theorem}\label{T:A.4}
For any initial measure $\mu\in \cM_\rho(\mR^d)$, the martingale problem \eqref{eq:MP} has a unique solution.
\end{theorem}
\begin{proof}
	Let $X$ be any solution to \eqref{eq:MP} with initial measure $\mu\in \cM_\rho(\mR^d)$. Following \cite{Myt1996environments}, one only need to show that for any $\phi\in C_c^+(\mR^d)$,
	$$
	\lim_{n\to\infty} \bE\left[\exp(-X_0(Y_t^{(n)}))\right] = \bE\left[\exp(-\<\phi, X_t\>)\right].
	$$
	By equation (4.18) in \cite{Myt1996environments}, we have
	\begin{align*}
		&\left|\bE\left[\exp(-\<\phi, X_t\>)\right] - \bE\left[\exp(-X_0(Y_t^{(n)}))\right]\right| \\
		\leq& \frac{1}{2}\int_0^t \bE\left[ \int_{\mR^d \times \mR^d}
		|\cC(x,y)-\cC_n(x,y)| Y_{t-s}^{(n)}(x)Y_{t-s}^{(n)}(y) X_s(\mathrm{d}x)X_s(\mathrm{d}y) \right] \mathrm{d}s \\
		&+ \frac{1}{6\sqrt{n}}\int_0^t \left| \bE\left[ \bE_{\xi^{(n)}} \left[X_s\left(\xi^{(n)}Y_{t-s}^{(n)}\right)^3 \right] \right]\right| \mathrm{d}s =: \cE_n^{(1)} + \cE_n^{(2)}.
	\end{align*}
We use $\bE_X$ and $\bE_Y$ to denote the expectation with respect to $X$ and $Y$, respectively.
	Using Lemma \ref{lem:dual_3moment}, the dominated convergence theorem and the fact that $\cC_n$ converges to $\cC$, we get
	\begin{align*}
		\cE_n^{(1)} \leq C \int_0^t \bE_X\left[ \int_{\mR^d \times \mR^d} |\cC(x,y)-\cC_n(x,y)| \phi_\rho(x)\phi_\rho(y) X_s\otimes X_s (\d x \times \d y)\right] \mathrm{d}s \to 0.
	\end{align*}
	For \(\cE_{n}^{(2)}\), using Minkowski's inequality,
we have
	\begin{align*}
		\cE_n^{(2)}\leq& \frac{1}{6\sqrt{n}} \int_0^t \bE \left[ X_s\left( \left\{ \bE_{\xi^{(n)}} \left[ |\xi^{(n)}|^3(Y^{(n)}_{t-s})^3\right] \right\}^{\frac{1}{3}} \right) \right]^{3} \d s \\
		\leq & \frac{1}{6\sqrt{n}} \int_0^t \bE [X_s(Y_{t-s}^{(n)})]^3 \d s\leq \frac{1}{6\sqrt{n}} \int_0^t \bE_X \left[ X_s \left( \left[ \bE_Y  (Y^{(n)}_{t-s})^3 \right]^{\frac{1}{3}} \right) \right]^3 \d s\\
		\leq & Cn^{-\frac{1}{2}} \int_0^t \bE \left[ X_s(\phi_\rho)\right]^3 \d s \to 0, \quad \mbox{as } n\to\infty,
	\end{align*}
where in the second inequality we used Lemma \ref{lem:moment-bound}, and in the last inequality we used Lemma \ref{lem:dual_3moment}.
	So we conclude that $\lim_{n\to\infty} \bE\left[\exp(-X_0(Y_t^{(n)}))\right] = \bE\left[\exp(-\<\phi, X_t\>)\right]$, which completes the proof of uniqueness.
\end{proof}

\section{Positivity and comparison principle for semilinear SPDEs}\label{Sec:App-SPDE}
\setcounter{equation}{0}
\renewcommand\theequation{B.\arabic{equation}}

In this section, we present the positivity and comparison principle for semilinear SPDEs driven by space-time Gaussian noise. The results are standard, but we give
some key proofs for completeness.

  Let \( (\Omega, \cF, \{\cF_t\}_{t\geq 0}, \bP) \) be a fixed filtered probability space satisfying the usual conditions,
 and let $W$ be
  a Gaussian field with zero mean that is white in time and colored in space with
  $$
    \bE   \left[ W(s,x)W(t,y) \right]=\cC (x,y)(s\wedge t) \mbox{ and } \|\cC\|_\infty<\infty.
  $$

  The stochastic integration $\int^t_0\int_{\mR^d} f(s, x)W(\d s, x)X_s(\d x)$ is well defined for any predictable measurable random field $f$ and measure $X$ satisfying
  $$
  \bE \int_0^t\d s \int_{\mR^d\times \mR^d} f(s,x)  \cC(x,y) f(s, y) X_s(\d x) X_s(\d y)<\infty,
  $$
  and
  \begin{eqnarray}\label{eq:Ito}
  	\begin{aligned}
  		&\bE  \left[ \left(\int^t_0\!\!\int_{\mR^d} f(s, x)W(\d s, x) X_s(\d x) \right)\, \left(\int^t_0 \int_{\mR^d}  g(s, x)W(\d s, x) X_s(\d x) \right) \right]  \\
  		=& \bE \int_0^t\d s \int_{\mR^d\times \mR^d} f(s,x)  \cC(x,y) g(s,y)X_s(\d x) X_s(\d y),
  	\end{aligned}
  \end{eqnarray}
  due to It\^o  isometry.

  Let \(T>0\). Consider the following stochastic partial differential equation on $[0,T] \times \mR^d$:
  \begin{equation}\label{eq:spde1}
  	\begin{cases}
  		\displaystyle
  		\partial_t u(t,x)=\frac{1}{2}\Delta u(t,x)+b(t,x,u(t,x))+ k(t,x)+\sigma(t,x, u(t,x)) \p_t W(t,x)\\
  		u(0,x)= f(x)\in L^2.
  	\end{cases}
  \end{equation}
An \(\cF_t\)-progressively measurable random field \(u\) whose trajectories belong to \(C([0,T]; L^2) \cap L^2([0,T]; H^1)\) \(\bP\)-a.s. is called a probabilistically strong, analytically weak solution to \eqref{eq:spde1} if it satisfies the following conditions:
\begin{enumerate}
	\item For every test function \(\phi \in C_c^\infty(\mR^d)\) and \(t \in [0,T]\),
    \[
        \int_0^t \int_{\mR^d} \left| b(s, x, u(s,x)) + k(s,x) \right| |\phi(x)| \, \d x \, \d s < \infty, \quad \bP\text{-a.s.}
    \]
    and
    \[
        \int_0^t \int_{\mR^d} \left| \sigma(s, x, u(s,x)) \phi(x) \right|^2 \, \d x \, \d s < \infty, \quad \bP\text{-a.s.}
    \]
	\item For every test function \(\varphi \in C_c^\infty(\mR^d)\), the following equality holds \(\bP\)-almost surely for all \(t \in [0,T]\)
\begin{equation*}
    (u(t),\varphi)-(f,\varphi) = \frac{1}{2}\int_0^t (u(s),  \Delta \varphi) \d s +\int_0^t (b(u(s)), \varphi) \d s+\int_0^t ( \sigma(u(s)) W(\d s), \varphi).
\end{equation*}
\end{enumerate}

\begin{assumption}\label{aspt:1}
  	Let \(T>0\). Suppose  $b,\sigma: [0,T]\times \mR^d\times \mR\times \Omega\to \mR$ and $k:  [0,T] \times \mR^d\times \Omega\to \mR$ are predictable satisfying  the following conditions:
  	\begin{enumerate}
  		\item  for any $t\in [0,T]$, $x\in\mR^d$ and $r_1, r_2\in \mR$,
  		\begin{align*}
  			(b(t,x,r_1)-b(t,x,r_2))(r_1-r_2)+|\sigma(t,x,r_1)-\sigma(t,x,r_2)|^2 \leq \lambda |r_1-r_2|^2 ,
  		\end{align*}
  		
  		\item for any fixed $t\in [0,T]$ and $x\in\mR^d$,   $r\rightarrow b(t,x,r)$ is a continuous function on $\mR$, and
  		\begin{align*}
  			rb(t,x,r)+|\sigma (t,x,r)|^2\leq \lambda(1+r^2),
  		\end{align*}
  		
  		\item $\bE\int_0^T \|k(t)\|_2^2 \d t <\infty$.
  	\end{enumerate}
  \end{assumption}

  \begin{lemma}\label{lem:spde}
	Let \(T>0\). Under {\bf Assumption} \ref{aspt:1},
	the following hold true.
	\begin{enumerate}[(i)]
		\item  Equation \eqref{eq:spde1} has a unique solution $u \in   L^2(\Omega; C([0,T]; L^2))\bigcap L^2([0,T]\times \Omega; H^1)$, and
		\begin{equation*}
			\bE  \Big[  \sup_{0\leq t\leq T}\|u(t)\|_2^2 \Big] +\bE  \int_0^T\|u(t)\|_{H^1}^2dt\leq C \left(1+\|f\|_2^2+\bE  \int_0^T \|k(t)\|_2^2 \d t\right),
		\end{equation*}
		where $C$ only depends on $\lambda$ and $T$.
	
		\item Assume that $b$ and $\sigma$ in additionally satisfy
		\begin{align}\label{eq:Lgrowth}
			rb(t,x,r)+|\sigma (t,x,r)|^2\leq \lambda r^2,
		\end{align}
		and that $f\in L^p$ and $\bE\int_0^T \|k(t)\|_p^p \d t <\infty$ for some $p\geq 2$. Then
		\begin{equation}\label{eq:u-Lp}
			\bE  \Big[  \sup_{0\leq t\leq T}\|u(t)\|_p^p \Big] +\bE  \int_0^T\|\nabla u(t)\|_{p}^p \d t\leq C \left(\|f\|_p^p+\bE  \int_0^T \|k(t)\|_p^p \d t\right).
		\end{equation}
		
		\item Assume in additionally that $b(t,x,0)\geq 0$ and $\sigma(t,x,0) =0$, and that $f\geq 0$ and $k\geq 0$. Then
		\begin{equation}\label{positive-eta}
			\bP \l( \hbox{for a.e. } x\in \mR^d,  ~ u(t,x)\geq 0 \hbox{ for every  } t\in [0,T] \r)  =1.
		\end{equation}
	\end{enumerate}
  \end{lemma}
  \begin{proof}
  	(i). The proof for (i) is standard, we refer the reader to \cite[Theorem 2.13 and Theorem 2.24]{Par2021stochastic}.
  	
  	\smallskip
  	
  	(ii). Suppose $u\in L^p(\Omega; C([0,T]; L^p))\bigcap L^p([0,T]\times \Omega; W^{1,p})$ and $u$ is a solution to \eqref{eq:spde1}, then using integration by parts and It\^o's formula,
  	\begin{equation*}
  		\begin{aligned}
  		  \d_t \int_{\mR^d} |u(t,x)|^p \d x = &- p(p-1)\left( \int_{\mR^d} |u(t,x)|^{p-2} |\nabla u(t,x)|^2 \d x\right) \d t\\
		  & + p\left(\int_{\mR^d} |u(t,x)|^{p-2} u(t,x)[ b(t,x,u(t))+k(t,x)] \d x \right)  \d t\\
  		  &+ \frac{p(p-1)}{2} \left(\int_{\mR^d} |u(t,x)|^{p-2} \sigma^2(t,x,u(t,x)) \cC(x,x) \d x\right) \d t\\
  		  &+ p \int_{\mR^d} |u(t,x)|^{p-2} u(t,x) \sigma(t,x,u(t,x)) W(\d t,\d x ).
  	\end{aligned}
  	\end{equation*}
   This together with \eqref{eq:Lgrowth}, H\"older inequality and Gronwall's inequality yields
  	\begin{equation*}
  		\sup_{t\in [0,T]} \bE \|u(t))\|_{p}^p + \bE \int_0^T \|\nabla u(t)\|_{p}^p \d t \leq C \|f\|_p^p+ C \bE\int_0^T \|k(t)\|_p^p \d t.
  	\end{equation*}
  	This can be refined to
	 satisfy
	 \eqref{eq:u-Lp}, due to  Burkholder-Davis-Gundy inequality (see the proof of \cite[Lemma 2.17]{Par2021stochastic}).
  	
   	(iii). The proof for (iii) is almost the same as that for  \cite[Theorem 2.24]{Par2021stochastic}. We give a proof here for
   	the sake of completeness.
   	
   	Consider the following  SPDE:
   	$$
   	  \partial_t \bar{u}=\frac{1}{2} \Delta \bar{u}+b(\bar{u}^+)+k+\sigma(\bar{u}^+) \p_t W
   	$$
   	with the initial data $\bar u(0)=f\geq 0$. We claim that
   	\begin{equation}\label{eq:positive}
   		\bP \l(  \hbox{for a.e. } x\in \mR^d ,  \,  \bar u(t,x)\geq 0 \hbox{ for every  } t>0 \r)=1.
   	\end{equation}
   	Assuming this, then by the uniqueness result (i), we have  $u=\bar{u}$, which gives
   	\eqref{positive-eta}. Thus  we are left to show \eqref{eq:positive}.

   	Let $\phi_\eps\in C^2(\mR)$ be a
	convex regularization of $x^2\1_{\{x<0\}}$  defined by
   	\[
   	\phi_\eps(x)=
   	\begin{cases}
   		0 & \mbox{ if }x\geq 0,\\
   		x^2-\eps^2/6 &\mbox{ if } x\leq -\eps,\\
   		-\eps^{-1}x^3 \left( (2\eps)^{-1}x+4/3 \right)  &\mbox{ if } -\eps< x< 0.
   	\end{cases}
   	\]
   	Put
   	$$
   	  \Phi_\eps(u)=\int_{\mR^d}\phi_\eps(u(x))\d x.
   	$$
   	By the generalized Ito's formula (see \cite[Lemma 2.15]{Par2021stochastic}), we have
   	\begin{equation}\label{4}
   		\begin{split}
   			\Phi_\eps(\bar{u}(t))=&-\frac{1}{2}\int_0^t\l( \nabla\bar{u}(s), \phi''_\eps(\bar u(s))\nabla \bar u(s)\r) \d s +  \int_0^t\left(b(\bar{u}^+(s))+k(s),\phi_\eps'(\bar{u}(s))\right) \d s\\
   			 &+\frac{1}{2}\int_0^t\int_{\mathbb{R}^d}\phi_\eps''(\bar{u}(s,x))\sigma(\bar{u}^+(s,x))^2\cC (x,x)\d x\d s\\
   			 &+\int_0^t\int_{\mR^d}\phi_\eps'(\bar{u}(s,x))\sigma(\bar{u}^+(s,x))W(\d s, x)\d x.
   		\end{split}
   	\end{equation}
   	Since $\phi_\eps'\leq 0$, $\phi_\eps''\geq 0$ and $k, f\geq 0$, we have
   	$$
   	  -\l( \nabla \bar{u}(s), \phi_\eps''(\bar{u}(s))\nabla \bar{u}(s) \r)\leq 0, \quad
   	  (k(s),\phi_\eps'(\bar{u}(s)))\leq0, \quad  (b(\bar{u}^+(s)),\phi_\eps'(\bar{u}(s)))\leq 0,
   	$$
   	and
   	$\phi_\eps'(\bar u(s))\sigma(\bar{u}^+(s))=0$ and  $\phi_\eps''(u(s))\sigma(\bar{u}^+(s))=0$. Taking expectation on both sides of (\ref{4}), we obtain that $\bE  \Phi_\eps(\bar{u}_t) \leq 0$. So
   	$$
   	  \bP (\Phi_\eps(\bar{u}_t)=0)=1, ~\mbox{ for each }t>0.
   	$$
   	Since $\Phi_\eps(\bar{u}_t)$ is continuous in $t$, in fact, we have
   	$$
   	  \bP (\Phi_\eps(\bar{u}_t)=0 \hbox{ for every }  t>0)=1.
   	$$
    This implies \eqref{eq:positive}.
   \end{proof}

  Now suppose $v$ satisfies a slightly stochastic partial different equation
  $$
    \partial_t v=\frac{1}{2}\Delta v +\widetilde{b}(v)+ \widetilde{k}+{\sigma}(v) \p_t W, \quad v(0)=\widetilde{f}.
  $$
 Adapting the proof of \cite[Theorem 2.26]{Par2021stochastic} and using arguments similar to Lemma \ref{lem:spde} (iii), we establish the following comparison theorem.

  \begin{lemma}\label{lem:comparion}
  	Assume that $f\leq \widetilde{f}$, $b\leq \widetilde{b}$ and $k\leq \widetilde{k}$, and that one of the pairs $(b, \sigma, k)$ and $(\widetilde{b}, \sigma, \widetilde{k})$ satisfies {\bf Assumption} \ref{aspt:1}. Then $u(t,x)\leq v(t,x)$ $x$ a.e., $\bP$-a.s., for all $t\geq 0$.
  \end{lemma}

  \begin{remark}\label{rmk:bdd-initial} \rm
      In the main body of this paper, for equation \eqref{eq:PAM} we also consider initial date $f\in C_b(\mR^d)$. The solution $v$ to \eqref{eq:PAM} in this context should be interpreted as follows: Suppose that  $C_c^+(\mR^d) \ni f_n^+ \uparrow f^+$ and $C_c^+(\mR^d) \ni f_n^- \uparrow f^-$. Let $v_n$ be the solutions to \eqref{eq:PAM} with $v_n(0)=f_n$.
      Using the linearity of \eqref{eq:PAM}, Lemma \ref{lem:comparion} and the estimates provided in Lemma \ref{lem:control-v}, one can see that $v_n$ convergence to a continuous random field $v$, which satisfies equation \eqref{eq:PAM} in  analytically weak and probabilistically strong sense. Moreover, the comparison principle also holds in this case.
  \end{remark}

  \section{Proof for Lemma \ref{lem:LDP}}\label{Sec:App-LDP}
  \setcounter{equation}{0}
  \renewcommand\theequation{C.\arabic{equation}}

  Recall that $W=\sqrt{a} \xi$.  We denote the expectation and probability with respect to the noise \(W\) (and \(\xi\)) by \(\bE\) and \(\bP\), respectively, whereas \(E\) and \(P\) are reserved for the Brownian motion \(B\), which is independent of \(W\) (and \(\xi\)). Put
  $$
    \eta_a(t,x)=\int_0^t \xi (\d s, x+\sqrt{1/a} B_{s}), \quad w_a (t,x) = E \l[ \exp (\eta_a(t,x) )\r].
  $$

  By the independence and stationarity of the increments of $\xi$ in time and reversing time in the\ stochastic integral in \eqref{eq:SFK}, one can see that
  \begin{equation}\label{eq:v=w}
      \begin{aligned}
    \widetilde{v}(t,\cdot)=& E\l[\exp \l( \int_0^t\sqrt{a}\xi(\d s,\cdot+B_{t-s})\r)\r]  \overset{d}{=} E\l[\exp\int_0^t \sqrt{a}\xi (\d s,\cdot+B_{s}) \r] \\
    \overset{d}{=}& E \l[\exp\int_0^t \xi({a}\d s,\cdot+B_{s})\r]
    \overset{d}{=} E\l[ \exp\int_0^{a t} \xi(\d u, \cdot+\sqrt{1/a}B_{u})\r]\\
    =&w_a(a t, \cdot) .
  \end{aligned}
  \end{equation}
  Thanks to \cite[Theorem 2]{CV1998almost}, there is a constant $N_0(d, \Theta)\geq 1$ such that for each $a>N_0$ and $x\in \mR^d$,
  \begin{equation}\label{eq:wa-Lyap}
      \lambda (a)=\lim_{t\to\infty} \frac{\log w_a (t,x)}{t}\leq \frac{C}{\log a}<\frac{1}{100}.
  \end{equation}

  \smallskip

  The following large deviation type result is proved by Cranston-Mountford \cite[Theorem 3.5]{CM2006lyapunov}.

  \const{\cE1}
  \begin{lemma}\label{lem:prob-wa}
    For $t$ sufficiently large, there is a constant $c_{\cE1}(d, a, \Theta)>0$ such that for each $x\in \mR^d$,
    \begin{equation}\label{eq:LDP1}
        \bP (w_a(t,x)>\e^{(\lambda(a)+\frac{1}{100})t}) \leq \e^{-c_{\cE1} t}.
    \end{equation}
  \end{lemma}

  To enhance the estimation of $w_a(t,x)$ in \eqref{eq:LDP1} to an estimation of $\sup_{|x|\leq t} w_a(t,x)$, we also need a Harnack-type inequality for $w_a$.

  \const{\cE2}
  \begin{lemma}\label{lem:harnack}
  There is a constant $c_{\cE2}$ only depends on $d$ and $\Theta$ such that
    \begin{equation*}
        \bP \left(w_a(t, x) > w_a(t, y)\e^{(t|x-y|)^{1 / 3}}+\e^{t/2-c_{\cE2}(t|x-y|)^{-1 / 3}}\right) \leqslant \e^{t/2-c_{\cE2}(t|x-y|)^{-1 / 3}} .
    \end{equation*}
  \end{lemma}
  \begin{proof}
    Let
      $$
        A:= \l\{ \l|\eta_a(t,x)-\eta_a(t,y)\r| >(t|x-y|)^{\frac{1}{3
        }} \r\}.
      $$
      Recall that $\Theta(0)=1$. Given
	  \(\{ B_s, s\geq 0\}\), $\eta_a(t,x)-\eta_a(t,y)$ is a Gaussian random variable with mean zero and variance
      $$
        \bE |\eta_a(t,x)-\eta_a(t,y)|^2= 2 t(1-\Theta(x-y)) \leq C t |x-y|, \quad \forall |x-y|\ll 1.
      $$
      Thus, there exists $c_{\cE2}=c_2(d, \Theta)>0$ such that
      $$
        \bP \l( A\r) \leq \e^{-4c_{\cE2}(t|x-y|)^{-1/3}},  \quad \forall |x-y|\ll 1.
      $$
      This yields
      \begin{equation}\label{eq:PP}
          \begin{aligned}
           \bE E \left[\e^{\eta_a(t,x)} \1_A\right] \leq& E \l[ \l(\bE  \e^{2\eta_a(t,x)} \r)^{\frac{1}{2}} \bP^{\frac{1}{2}} (A)\r]
          \leq \l(E \bE  \e^{2\eta_a(t,x)} \r)^{\frac{1}{2}} \e^{-2c_{\cE2}(t|x-y|)^{-1/3}}.
      \end{aligned}
      \end{equation}
      Given $\{ B_s, s\geq 0\}$, $2\eta_a(t,x)$ is a Gaussian random variable with mean zero and variance $\bE  |2\eta_a(t,x)|^2 = 4t$, so $\bE E \left[\e^{2\eta_a(t,x)}\right] = E \bE \left[\e^{2\eta_a(t,x)}\right]\leq \e^{2t}$. Combining this with Chebyshev's inequality and \eqref{eq:PP}, we obtain
      \begin{equation*}
          \begin{aligned}
            \bP \l( E \left[\e^{\eta(t,x)} \1_A\right] > \e^{t/2-c_{\cE2}(t|x-y|)^{-1/3}} \r)
			&\leq \bE E \left[\e^{\eta_a(t,x)} \1_A\right] \e^{-\frac{t}{2}+c_{\cE2}(t|x-y|)^{-1/3}}\\
            &\overset{\eqref{eq:PP}}{\leq} \e^{t/2-c_{\cE2}(t|x-y|)^{-1/3}}.
          \end{aligned}
      \end{equation*}
      Therefore, with probability at least $1- \e^{t/2-c_{\cE2}(t|x-y|)^{-1/3}}$,
      \begin{equation*}
          \begin{aligned}
            w_a(t, x) & =E \left[\e^{\eta_a(t,x)} \1_A\right]+E\left[\e^{\eta_a(t, x)} \1_{A^c}\right]
            \leq \e^{t/2-c_{\cE2}(t|x-y|)^{-1/3}} +E\left[ \e^{\eta_a(t,y)} \right]\e^{(t|x-y|)^{\frac{1}{3}}} \\
            & \leq \e^{t/2-c_{\cE2}(t|x-y|)^{-1/3}} +
           w_a(t,y)\e^{(t|x-y|)^{\frac{1}{3}}}.
          \end{aligned}
      \end{equation*}
  So we complete the proof.
  \end{proof}

  A chaining argument and Lemma \ref{lem:prob-wa} give us Lemma \ref{lem:LDP}.

  \const{\cEE}
  \begin{proof}[Proof for Lemma \ref{lem:LDP}]
  Set $D_i:= 2^{-i}t^{-10} \mZ^d \cap [-t,t]^d$, $i\in \mN$. Let
  $$
    A_0:= \l\{w_a(t,x)\leq \e^{(\lambda(a)+\frac{1}{100})t}, ~ x\in D_0\r\},
  $$
  and
  \begin{equation*}
    \begin{aligned}
    A_i:= & \Big\{\forall(x, y) \in\left(D_i \backslash D_{i-1}\right) \times D_{i-1},|x-y|=t^{-10} 2^{-i}, \\
    & w_a(t, x) \leqslant w_a(t, y) \e^{(t|x-y|)^{1 / 3}}+\e^{t/2-c_{\cE2}
    (t|x-y|)^{-1 / 3}}\Big\} .
    \end{aligned}
  \end{equation*}
  By Lemma \ref{lem:prob-wa}, we get
  \begin{equation*}
      \bP (A_0^c) \leq |D_0| \e^{-c_1 t}\leq t^{100 d}\e^{-c_1 t}.
  \end{equation*}
  We also have
  \begin{equation*}
      \bP (A_i^c) \leq 2^{i+1} |D_0| \exp \l( \frac{t}{2}-c_2t^{3}2^{i/3}\r)~ (i=1,2,\cdots),
  \end{equation*}
  due to Lemma \ref{lem:harnack}.
  Thus, there exists a constant \(c_{\cEE}>0\), depending only on \(d\) and \(\Theta\), such that
  \begin{equation}\label{eq:probA}
      \bP \l(\bigcup_{i=0}^\infty A_i^c\r)\leq C t^{100d} \l(\e^{-c_1 t}+ \e^{t/2} \sum_{i=1}^\infty 2^i \exp\l(- c_2 t^3 2^{i/3}\r)\r) \leq \e^{-c_{\cEE}t}.
  \end{equation}
  On $\bigcap_{i=0}^\infty A_i$, for each $|x|\leq t$, there is a sequence $x_i\in D_i$ such that $x_i\to x$ and $|x_i-x_{i-1}|\leq \sqrt{d}2^{-i} t^{-10}$, and
  \begin{equation}\label{eq:iterate}
      w_a(t,x_0)\leq \e^{(\lambda(a)+\frac{1}{100})t}~ \mbox{ and }~  w_a(t,x_i) \leq w_a(t,x_{i-1})\e^{t^{-3} 2^{-i/3}}+\e^{t/2- c_2 t^3 2^{i/3}}.
  \end{equation}
  Iterating \eqref{eq:iterate}, on $\bigcap_{i=0}^\infty A_i$, for sufficiently large $t$ and each $|x|\leq t$, it holds that
  \begin{equation*}
    \begin{aligned}
    w_a(t,x) &=\lim_{x_i\to x} w_a(t,x) \leq \lim_{n \to \infty} \left( w_a(t, x_0) \prod_{j=1}^n \e^{t^{-3} 2^{-j/3}} + \sum_{j=1}^n \e^{\frac{t}{2} - c_2 t^3 2^{j/3}} \prod_{k=j+1}^n \e^{t^{-3} 2^{-k/3}} \right) \\
    &\leqslant \e^{(\lambda(a)+\frac{1}{100})t +t^{-3} \sum_{j=1}^\infty 2^{-j / 3}}+\e^{t/2} \sum_{j=1}^\infty \e^{-c_2 t^{3} 2^{j / 3}+t^{-3} \sum_{k=j+1}^\infty 2^{-k / 3}} \\
    & \leq \e^{(\lambda(a)+\frac{1}{50})t} + \e^{t / 2} \sum_{j=1}^\infty \e^{-c_2 t^3 2^{j / 3}+C t^{-3}}
     \leq \e^{(\lambda(a)+\frac{1}{10})t}.
    \end{aligned}
  \end{equation*}
Therefore,
  \begin{equation}\label{eq:LDP-wa}
      \bP \l(\sup_{|x|\leq t} w_a(t,x) > \e^{(\lambda(a)+\frac{1}{10})t} \r) \leq \bP \l(\bigcup_{i=0}^\infty A_i^c\r) \overset{\eqref{eq:probA}}{\leq} \e^{-c_{\cEE} t}.
  \end{equation}
  Noting that $a\geq N_0\geq 1$, and utilizing \eqref{eq:v=w}, \eqref{eq:wa-Lyap} and \eqref{eq:LDP-wa}, we get
  \begin{align*}
      \bP \l(\sup_{|x|\leq t} v(t,x)> \e^{-\frac{a t}{3}}\r) =& \bP \l(\sup_{|x|\leq t} \widetilde{v}(t,x)> \e^{\frac{a t}{6}}\r) \overset{\eqref{eq:v=w}}{=} \bP \l(\sup_{|x|\leq t} w_a(at,x)> \e^{\frac{a t}{6}}\r) \\
      \overset{\eqref{eq:wa-Lyap}}{\leq} &\bP \l(\sup_{|x|\leq at} w_a(at,x)> \e^{(\lambda(a)+\frac{1}{10}) at}\r) \overset{\eqref{eq:LDP-wa}}{\leq} \e^{-c a t}.
  \end{align*}
So we complete the proof.
      \end{proof}

\end{document}